\documentclass[a4paper,12pt]{amsart}
\usepackage{amssymb,amscd,amsmath}
\usepackage[all]{xy}
\usepackage[dvips]{graphicx}

\newcounter{mytheorem} 

\newtheorem{theorem}[mytheorem]{Theorem}
\newtheorem{definition}[mytheorem]{Definition}
\newtheorem{lemma}[mytheorem]{Lemma}

\newtheorem{example}[mytheorem]{Example}

\newtheorem{remark}[mytheorem]{Remark}
\newtheorem{corollary}[mytheorem]{Corollary}
\newtheorem{claim}[mytheorem]{Claim}
\newtheorem*{clm*}{Claim}

\makeatletter
\@addtoreset{subsection}{section}
\@addtoreset{subsubsection}{section}
\@addtoreset{subsubsection}{subsection}
\makeatother

\makeatletter
\newcommand\testshape{family=\f@family; series=\f@series; shape=\f@shape.}
\def\myemphInternal#1{\if n\f@shape%
\begingroup\itshape #1\endgroup\/%
\else\begingroup\bfseries #1\endgroup%
\fi}
\def\myemph{\futurelet\testchar\MaybeOptArgmyemph}
\def\MaybeOptArgmyemph{\ifx[\testchar \let\next\OptArgmyemph
                 \else \let\next\NoOptArgmyemph \fi \next}
\def\OptArgmyemph[#1]#2{\index{#1}\myemphInternal{#2}}
\def\NoOptArgmyemph#1{\myemphInternal{#1}}
\makeatother

\newcommand\RRR{\mathbb{R}}
\newcommand\CCC{\mathbb{C}}
\newcommand\ZZZ{\mathbb{Z}}
\newcommand\NNN{\mathbb{N}}
\newcommand\FFF{\mathbb{F}}

\newcommand\Int{\mathrm{Int}}

\newcommand\Per{\mathrm{Per}}

\newcommand\id{\mathrm{id}}
\newcommand\cl[1]{\overline{#1}}

\newcommand\defeq{:=}

\newcommand\AFld{F}
\newcommand\AFlow{\mathbf{\AFld}}
\newcommand\BFld{G}
\newcommand\BFlow{\mathbf{\BFld}}

\newcommand\HFld{H}
\newcommand\HFlow{\mathbf{\HFld}}

\newcommand\Mman{M}
\newcommand\afunc{\alpha}
\newcommand\bfunc{\beta}

\newcommand\mufunc{\mu}
\newcommand\amap{f}

\newcommand\dd{d}

\newcommand\Nman{N}

\newcommand\Qman{Q}
\newcommand\Vman{V}
\newcommand\Uman{U}
\newcommand\Wman{W}

\newcommand\wfpt[1]{\widehat{#1}}
\newcommand\VmanS{\wfpt{\Vman}}
\newcommand\WmanS{\wfpt{\Wman}}

\newcommand\domA{\mathsf{dom}(\AFlow)}

\newcommand\domB{\mathsf{dom}(\BFlow)}

\newcommand\funcA{\mathsf{func}(\AFlow)}
\newcommand\funcAV{\mathsf{func}(\AFlow,\Vman)}

\newcommand\funcAM{\mathsf{func}(\AFlow,\Mman)}

\newcommand\funcB{\mathsf{func}(\BFlow)}

\newcommand\ShA{\varphi}

\newcommand\ShAV{\ShA_{\Vman}}
\newcommand\ShAW{\ShA_{\Wman}}

\newcommand\ShAVS{\ShA_{\VmanS}}
\newcommand\ShAWS{\ShA_{\WmanS}}

\newcommand\ShAM{\ShA_{\Mman}}

\newcommand\kersh{\ker}
\newcommand\ZidA{\kersh(\ShA)}
\newcommand\ZidAV{\kersh(\ShAV)}
\newcommand\ZidAW{\kersh(\ShAW)}

\newcommand\ZidAVS{\kersh(\ShAVS)}
\newcommand\ZidAWS{\kersh(\ShAWS)}

\newcommand\eps{\varepsilon}

\newcommand\EAFlow{\mathcal{E}(\AFlow)}
\newcommand\DAFlow{\mathcal{D}(\AFlow)}
\newcommand\EidAFlow[1]{\mathcal{E}_{\id}(\AFlow)^{#1}}
\newcommand\DidAFlow[1]{\mathcal{D}_{\id}(\AFlow)^{#1}}
\newcommand\DiffM{\mathcal{D}(\Mman)}

\newcommand\EBFlow{\mathcal{E}(\BFlow)}

\newcommand\EidBFlow[1]{\mathcal{E}_{\id}(\BFlow)^{#1}}

\newcommand\EidAV[1]{\mathcal{E}_{\id}(\AFlow,\Vman)^{#1}}

\newcommand\EAV{\mathcal{E}(\AFlow,\Vman)}
\newcommand\EAM{\mathcal{E}(\AFlow,\Mman)}

\newcommand\EidAW[1]{\mathcal{E}_{\id}(\AFlow,\Wman)^{#1}}

\newcommand\Wtop{\mathsf{W}}

\newcommand\Wr[1]{\Wtop^{#1}}

\newcommand\Cnt[1]{\mathcal{C}}
\newcommand\Cont[1]{\mathcal{C}^{#1}}
\newcommand\Cinf{\Cont{\infty}}
\newcommand\Cr[3]{\Cont{#1}(#2,#3)}
\newcommand\Ci[2]{\Cr{\infty}{#1}{#2}}

\newcommand\orb{o}

\newcommand\FixF{\Sigma}
\newcommand\FixA{\FixF_{\AFld}}
\newcommand\FixB{\FixF_{\BFld}}

\newcommand\dfrmc[2]{(#1,#2)} 
\newcommand\dfrm[3]{(#1;#2,#3)}

\newcommand\ESDL[3]{\mathsf{ESD}(#1;#2,#3)}
\newcommand\EP[1]{(E)^{#1}}

\newcommand\GV{\Gamma^{+}_{\Vman}}
\newcommand\imGV{\ShAV(\GV)}

\newcommand\imSh[1]{Sh(#1)}
\newcommand\imShB{\imSh{\BFlow}}
\newcommand\imShAV{\imSh{\AFlow,\Vman}}

\newcommand\imShAM{\imSh{\AFlow,\Mman}}
\newcommand\imShA{\imSh{\AFlow}}

\newcommand\imShAW{\imSh{\AFlow,\Wman}}

\newcommand\PN{{\sf pn}}

\newcommand\dif{h}

\newcommand\orig{\mathsf{0}}

\newcommand\om{\orig_{m}}

\newcommand\on{\orig_{n}}

\newcommand\omn{\orig_{m+n}}

\newcommand\Vmanm{\Vman^{m}}
\newcommand\Vmann{\Vman^{n}}
\newcommand\Vmanmn{\Vman^{m+n}}
\newcommand\EAVmn{\mathcal{E}(\AFlow,\Vmanmn)}
\newcommand\EBVm{\mathcal{E}(\BFlow,\Vmanm)}
\newcommand\ShAVmn{\ShA_{\Vmanmn}}

\newcommand\Wmanm{\Wman^{m}}
\newcommand\Wmann{\Wman^{n}}

\newcommand\EBWm{\mathcal{E}(\BFlow,\Wmanm)}

\newcommand\mOmega{\Psi}

\newcommand\sx{\sigma}
\newcommand\cx{\tau}
\newcommand\pp{\rho}
\newcommand\tim{t}

\newcommand\Subsp{\mathcal{H}}
\newcommand\SmParMan{S}

\newcommand\ParMan{P}
\newcommand\CntParMan{T}
\newcommand\DestMan{U}

\newcommand\regext[1]{#1}
\newcommand\LT{{\rm(L)}}
\newcommand\RH{{\rm(H)}}
\newcommand\RHE{{\rm(HE)}}
\newcommand\RHS{{\rm(HS)}}
\newcommand\reLT{{\regext{\LT}}}
\newcommand\reRH{{\regext{\RH}}}
\newcommand\reRHE{{\regext{\RHE}}}
\newcommand\reRHS{{\regext{\RHS}}}

\newcommand\GSF{\mathsf{GSF}}

\newcommand\VFM{\mathcal{V}(\Mman)}
\newcommand\CMR{\Ci{\Mman}{\RRR}}
\newcommand\prs[1]{\langle#1\rangle}
\newcommand\fleq{\preccurlyeq}

\newcommand\RPF[1]{RP(\Vman)}

\newcommand\vfunc{v}
\newcommand\wfunc{w}
\newcommand\Jord{\mathbf{J}}

\newcommand\org{\cx_0}
\newcommand\ld{l}

\begin{document}
\title[Image of a shift map along the orbits of a flow]{Image of a shift map along \\ the orbits of a flow}
\author{Sergiy Maksymenko}
\address{Topology dept., Institute of Mathematics of NAS of Ukraine, Tere\-shchenkivs'ka st. 3, Kyiv, 01601 Ukraine}
\date{5/11/2009}
\email{maks@imath.kiev.ua}
\keywords{orbit preserving diffeomorphism, shift map}
\subjclass{37C10}

\begin{abstract}
Let $(\textbf{F}_t)$ be a smooth flow on a smooth manifold $M$ and $h:M\to M$ be a smooth orbit preserving map.
The following problem is studied: suppose that for every $z\in M$ there exists a germ $\alpha_z$ of a smooth function at $z$ such that $h(x)=F_{\alpha(x)}(x)$ near $z$; can the germs $(\alpha_z)_{z\in M}$ be glued together to give a smooth function on all of $M$?
This question is closely related to reparametrizations of flows.
We describe a large class of flows $(\textbf{F}_t)$ for which the above problem can be resolved, and show that they have the following property: any smooth flow $(\textbf{G}_t)$ whose orbits coincides with the ones of $(\textbf{F}_t)$ is obtained from $(\textbf{F}_t)$ by smooth reparametrization of time.
\end{abstract}

\maketitle

\section{Introduction}\label{sect:intro}
Let $\Mman$ be a smooth $(\Cinf)$, connected, $m$-dimensional manifold possibly non-compact and with or without boundary, $\AFld$ be a smooth vector field on $\Mman$ tangent to $\partial\Mman$ and generating a flow $\AFlow:\Mman\times\RRR \to \Mman$, and $\FixF$ be the set of singular points of $\AFld$.
For $x\in\AFld$ we will denote by $\orb_x$ the orbit of $x$ and if $x$ is periodic, then $\Per(x)$ is the period of $x$.

Let $\DiffM$ be the group of all $\Cinf$ diffeomorphisms of $\Mman$ and $\DAFlow$ be the subgroup of $\DiffM$ consisting of all \myemph{diffeomorphisms of $\AFlow$}, that is diffeomorphisms $\dif:\Mman\to\Mman$ such that $\dif(\orb)=\orb$ for every orbit of $\AFlow$.
A natural problem, which usually appears in studying functional spaces, is to \myemph{find general formulas for elements of $\DAFlow$}, or \myemph{parametrize $\DAFlow$ with elements of a certain space which seems to be simpler}, e.g. with functions.
In general this question is very difficult.
However, if we confine ourselves with the path component of $\DAFlow$ with respect to some natural topology, then in many cases the mentioned problem can be satisfactory resolved in terms of the flow of $\AFld$.

The aim of the present paper is to give sufficient conditions when the ``\myemph{local parametrizations of elements of $\DAFlow$ via functions can be glued to a global one}'', (Theorem\;\ref{th:suffcond_for_imShA_EidAk}). 
We also present a class of vector fields satisfying those conditions (Theorem\;\ref{th:LT_RH}).

We will now explain the meaning of the previous paragraph.
First of all it is more convenient to extend $\DAFlow$ and work with the subsemigroup $\EAFlow$ of $\Ci{\Mman}{\Mman}$ consisting of maps $\dif:\Mman\to\Mman$ such that 
\begin{enumerate}
\item
$\dif(\orb) \subset\orb$ for every orbit $\orb$ of $\AFld$;
 \item 
$\dif$ is a local diffeomorphism at every singular point $z\in\FixF$.
\end{enumerate}
We will call $\EAFlow$ the semigroup of \myemph{endomorphisms} of $\AFld$.
Evidently, $\DAFlow = \EAFlow\cap\DiffM$.

For $0\leq k\leq \infty$ denote by $\EidAFlow{k}$ (resp. $\DidAFlow{k}$) the path component of the identity map  $\id_{\Mman}$ in $\EAFlow$ (resp. $\DAFlow$) in the weak topology $\Wr{k}$, see \S\ref{sect:deformations}.
It follows that $\EidAFlow{k}$ consists of all maps $\dif\in\EAFlow$ which are homotopic to $\id_{\Mman}$ in $\EAFlow$ via a homotopy which \myemph{induces a homotopy on the level of $k$-jets}.
In particular, $\EidAFlow{0}$ consists of all $\dif\in\EAFlow$ homotopic to $\id_{\Mman}$ in $\EAFlow$.
Moreover, $\EidAFlow{k}$ for $k\geq1$ \myemph{contains (but does not coincide with)} the space of all $\dif\in\EAFlow$ homotopic to $\id_{\Mman}$ in $\EAFlow$ via some $\Cont{k}$ homotopy.
Similar descriptions hold for $\DidAFlow{k}$.

Define also the following map $\ShA:\Ci{\Mman}{\RRR}\to\Ci{\Mman}{\Mman}$ by 
$$
\ShA(\afunc)(x)=\AFlow(x,\afunc(x)), \qquad \afunc\in\Ci{\Mman}{\RRR},
$$
which will be called $\ShA$ the \myemph{shift map} along the orbits of $\AFld$, see\;\cite{Maks:TA:2003}%
\footnote{I must warn the reader that my paper \cite{Maks:TA:2003} contains some misprints and gaps.
However, they are not ``dangerous'' for the present paper, see Remark\;\ref{rem:Maks:TA:2003-gaps} below.}.
The map $\dif=\ShA(\afunc):\Mman\to\Mman$ will be called the \myemph{shift} via $\afunc$, which in turn will be a \myemph{shift function} for $\dif$.
Denote the image of $\ShA$ in $\Ci{\Mman}{\Mman}$ by $\imShA$.

\begin{lemma}\label{lm:sh_loc_diff}{\cite[Lm.\;20 \& Cor.\;21]{Maks:TA:2003}}
Let $\afunc\in\Ci{\Mman}{\RRR}$ and $z\in\Mman$.
Then $\ShA(\afunc)$ is a local diffeomorphism at $z$ iff and only if  $\AFld(\afunc)(z)\not=-1$, where $\AFld(\afunc)$ is the Lie derivative of $\afunc$ along $\AFld$.
In particular, this holds for each $z\in\FixF$, as $\AFld(\afunc)(z)=0\not=-1$.
\end{lemma}

Finally, consider the following subset of $\Ci{\Mman}{\RRR}$:
$$\Gamma^{+}:=\{ \afunc\in\Ci{\Mman}{\RRR} \ : \ \AFld(\afunc)>-1\}.$$

\begin{lemma}\label{lm:imShA_EidAV}{\rm\cite[Lm.\;3.6]{Maks:MFAT:2009}}
The following inclusions hold true
\begin{equation}\label{equ:inclusions_imSh_EA}
\begin{array}{c}
 \imShA \;\subset\; \EidAFlow{\infty} \;\subset\; \cdots \;\subset\; \EidAFlow{1} \;\subset\; \EidAFlow{0}, \\ [2mm]
\imGV \;\subset\; \DidAFlow{\infty} \;\subset\; \cdots \;\subset\; \DidAFlow{1} \;\subset\; \DidAFlow{0}.
\end{array}
\end{equation}
If $\DidAFlow{k} \subset \imShA$ for some $k=0,\ldots,\infty$, e.g. if $\imShA = \EidAFlow{k}$, then $\ShA(\Gamma^{+}) = \DidAFlow{k}$.
\end{lemma}

The identity $\imShA = \EidAFlow{k}$ means that for each $\dif\in\EAFlow$ being homotopic to $\id_{\Mman}$ \myemph{via a homotopy in $\EAFlow$ inducing a homotopy of $k$-jets} there exists a smooth function $\afunc_{\dif}:\Mman\to\RRR$ such that 
\begin{equation} \label{equ:h_F_x_ax}
\dif(x)=\AFlow(x,\afunc_{\dif}(x)),
\end{equation}
for all $x\in\Mman$.
Then the representation\;\eqref{equ:h_F_x_ax} can be regarded as \myemph{general formulas for elements of $\EidAFlow{r}$} mentioned above.

In general, every $\dif\in\EidAFlow{0}$ has a $\Cinf$ shift function $\lambda$ defined at least on $\Mman\setminus\FixF$.
This $\lambda$ can be uniquely reconstructed from a particular homotopy between $\dif$ and $\id_{\Mman}$ in $\EAFlow$, but it may depend on such a homotopy and does not always extend to a smooth function on all of $\Mman$, see\;\cite[Th.\;25]{Maks:TA:2003} and Theorem\;\ref{th:shift-function-for-deformations}.
In fact the following statement holds true: if $\FixF=\varnothing$, then $\imShA=\EidAFlow{0}$ and the map
$$\ShA:\Ci{\Mman}{\RRR} \ \longrightarrow \ \imShA = \EidAFlow{0} $$
satisfies a covering path axiom.

Suppose $\FixF\not=\varnothing$.
It is shown in\;\cite{Maks:TA:2003} (see Corollary\;\ref{cor:regext_lin_imSh_EidA0}) that if $\AFld$ is linearizable (or more generally a ``regular extension'' of a linear vector field) at each of its singular points, then $\imShA=\EidAFlow{0}$.
Moreover, in\;\cite{Maks:hamv2,Maks:CEJM:2009} there were given examples of vector fields $\AFld$ on $\RRR^2$ for which $\imShA=\EidAFlow{1}$ and that for some of them $\EidAFlow{1} \not= \EidAFlow{0}$, see Lemma\;\ref{lm:Ed_and_ESDL_for_RH}.
The results of the present paper arise from analysis of those examples.

\smallskip

In this paper we consider the problem of gluing local shift functions to a global one.
More precisely, let $\dif\in\EAFlow$ and suppose that for each $y\in\FixF$ there exists a neighbourhood $\Vman_y$ and a $\Cinf$ function $\afunc_y:\Vman_y\to\RRR$ such that\;\eqref{equ:h_F_x_ax} holds for all $x\in\Vman_y$.
\myemph{Does there exist $\afunc\in\Ci{\Mman}{\RRR}$ such that\;\eqref{equ:h_F_x_ax} holds on all of $\Mman$}, that is \myemph{whether $\dif\in\imShA$}?

Due to Lemma\;\ref{lm:imShA_EidAV} it is necessary that $\dif\in\EidAFlow{k}$ for some $k\geq0$.
Our main results (Theorems\;\ref{th:main_result} and \ref{th:suffcond_for_imShA_EidAk}) claim that if $k\geq1$ and some rather general assumptions on singular points of $\AFld$ hold true, then $\dif\in\imShA$.
Theorems\;\ref{th:main_result} is based on results of M.\;Newman,  A.\;Dress, D.\;Hoffman, and L.\;N.\;Mann\;\cite{Dress:Topol:1969} about lower bounds of diameters of $\ZZZ_{p}$-actions on topological manifolds.

In \S\ref{sect:examples} we present examples of vector fields satisfying assumptions of Theorems\;\ref{th:main_result}.
We also prove that such vector field are in some sense ``maximal with respect to reparametrizations'', see Theorem\;\ref{th:LT_RH}.

\begin{remark}\rm 
The idea of substituting a function into a flow-map instead of time is not new, see e.g.\;\cite{EHoph:1937, Chacon:JMM:1966, Totoki:MFCKUS:1966, Kowada:JMSJ:1972, Parry:JLMS:1972, Kochergin:IANSSSR:1973, KondratevSamovol:MZ:1973}, where reparametrizations of measure preserving flows and mixing properties of such flows are studied.
In particular, in those papers the values of such functions and orbit preserving maps on sets of measure zero were ignored, so they allowed to be even discontinuous.

There are also analogues of this approach for discrete dynamical systems on Cantor set, e.g.\;\cite{GiordanoPutnamSkau:JRAM:1995,GiordanoPutnamSkau:IJM:1999,BezuglyiMedynets:CM:2008}.

In the present paper we consider smooth orbit preserving maps and require smoothness of their shift functions.
This turned out to be useful for the study of homotopical properties of groups of orbit preserving diffeomorphisms for vector fields, and stabilizers and orbits of certain classes of smooth functions with respect to actions of diffeomorphism groups, see\;\cite{Maks:TA:2003,Maks:AGAG:2006,Maks:BSM:2006,Maks:loc-inv-shifts}.
The results of this paper will be used to extend\;\cite{Maks:AGAG:2006}.
\end{remark}

\begin{remark}\label{rem:Maks:TA:2003-gaps}\rm
As noted in the introduction, the paper\;\cite{Maks:TA:2003} contains some misprints and gaps.
However they do not impact on the results of the present paper.

1) The \myemph{proof} of \cite[Pr.\;10]{Maks:TA:2003}, used for \cite[Th.\;12]{Maks:TA:2003} (see Theorem\;\ref{th:ker2}), is incorrect: it was wrongly claimed that \myemph{the map $\Psi(x,*):\mathcal{I}\to GL_{n}(\RRR)$, associating to each $t$ the Jacobi matrix of the flow map $\Phi_t$ at $x$, is a homomorphism}.
Nevertheless the \myemph{statement} of \cite[Pr.\;10]{Maks:TA:2003} is true, and reparations and extensions of that proposition are given in\;\cite{Maks:period_functions}.
In particular, Theorem\;\ref{th:ker2} is valid. 

2) \cite[Defn.\;24]{Maks:TA:2003} should contain an assumption that $\amap_t$ is a local diffeomorphism at every $z\in\FixF\cap\Vman$ (i.e. belongs to $\EAV$).
This assumption was explicitly used in \cite[Th.\;27 \& Lm.\;31]{Maks:TA:2003} but it was not verified in \cite[Lm.\;29]{Maks:TA:2003} so the proof of that lemma requires simple additional arguments.
We provide them in 1) of Lemma\;\ref{lm:pmOmega_is_deformation}.

3) \cite[Eq.\;(10)]{Maks:TA:2003} is misprinted and must be read as follows:
$$\alpha(x) = p_1\circ f(x) - p_1\circ \Phi(x,a)+a.$$
See Eq.\;\eqref{equ:local_formulas_for_shift_func} for another variant of this formula.

4) Assumptions of \cite[Defn.\;15]{Maks:TA:2003} are not enough for the proof of \cite[Th.\;17]{Maks:TA:2003}.
In fact in\;\cite[Defn.\;15]{Maks:TA:2003} it should be additionally required (in the notations of that definition) that $\ShAV(\mathcal{M})$ is open in $\ShAV(\Ci{V}{\RRR})$ with respect to some topology $\Wr{r}$.
Such an assumption was explicitly used in\;\cite[Th.\;17]{Maks:TA:2003}.
This implies useless of \cite[Lm.\;28]{Maks:TA:2003} being the part (B) of \cite[Th.\;27]{Maks:TA:2003} concerning (S)-points and verifying \cite[Defn.\;15]{Maks:TA:2003} for linear vector fields, and also incorrectness of the part of \cite[Th.\;1]{Maks:TA:2003} claiming that the shift map $\ShA$ is either a homeomorphism or a covering map.

5) The ``division'' lemma\;\cite[Lm.\;32]{Maks:TA:2003} also is not true.
Let $\FFF$ be either the field $\RRR$ or $\CCC$, $\Vman$ be an open neighbourhood of the origin in $\FFF$, and $Z:\Ci{\Vman}{\FFF}\to\Ci{\Vman}{\FFF}$ be the multiplication by $z$ map, that is $Z(\afunc)(z) = z\afunc(z)$ for $\afunc\in\Ci{\Vman}{\FFF}$.
Evidently, $Z$ is an injective map.
It was \myemph{wrongly claimed} in\;\cite[Lm.\;32]{Maks:TA:2003} that the inverse of $Z$ is continuous between topologies $\Wr{r}$ for all $r$.
In fact, it follows from the Hadamard lemma, that in the case $\FFF=\RRR$ the map $Z^{-1}$ is continuous from $\Wr{r+1}$ topology to $\Wr{r}$ for all $r\geq0$, while in the case $\FFF=\CCC$ I can only prove that $Z^{-1}$ is continuous only between $\Wr{\infty}$ topologies.
This implies that the estimations of the continuity of the local inverses of shift maps are incorrect, see paragraph after\;\cite[Eq.\;(26)]{Maks:TA:2003}.

Incorrect statements mentioned in 4) and 5) are not used in the present paper.
Their reparation is given in\;\cite{Maks:loc-inv-shifts}.

All other results in \cite{Maks:TA:2003} are correct.
In particular, besides Theorem\;12 and Eq.\;(10) we will use only the following ``safe'' statements from\;\cite{Maks:TA:2003}: \cite[Lm.\;5\;\&\;7]{Maks:TA:2003} (see Lemma~\ref{lm:ker1}), \cite[Lm.\;20 \& Cor.\;21]{Maks:TA:2003} (see Lemma\;\ref{lm:sh_loc_diff}), \cite[Th.\;25]{Maks:TA:2003} (it will be extended in Theorem\;\ref{th:shift-function-for-deformations}), \cite[Lm.\;29]{Maks:TA:2003} (it will also be extended in Theorem\;\ref{lm:inher_ESDL} of the present paper), and \cite[Eqs.\;(23)-(26)]{Maks:TA:2003}, see\;\S\S\ref{sect:shift-func-linear-R1}-\ref{sect:shift-func-linear-rotation-R2}.
\end{remark}

\subsection{Structure of the paper}
In next section we recall the definition of Whitney topologies and introduce certain types of deformations.
Then in \S\ref{sect:shift-map} we define shift maps of a vector field for open subsets $\Vman\subset\Mman$ and recall their properties established in\;\cite{Maks:TA:2003, Maks:period_functions}.
In \S\ref{sect:shift-func-at-reg} we prove Theorem\;\ref{th:shift-function-for-deformations} about existence of shift functions on the set of regular points of a vector field.

Further in \S\ref{sect:main-results} we prove Theorem\;\ref{th:main_result} being the main result of the present paper and deduce from it certain sufficient conditions when $\ShA=\EidAFlow{k}$ for some $k\geq0$ (Theorem\;\ref{th:suffcond_for_imShA_EidAk}).
 
In \S\ref{sect:desreasing_V} we show that ``\myemph{the identity $\imShA=\EidAFlow{k}$ is preserved if we properly decrease our manifold}'', see Lemma\;\ref{lm:reduce_imShV_EidAVk}.
This statement will not be used but it illustrates local nature of assumptions of Theorem\;\ref{th:suffcond_for_imShA_EidAk}.

Further in \S\ref{sect:regular-extensions} we consider regular extensions of vector fields and prove a lemma which, in particular, accomplish some missed arguments of\;\cite[Lm.\;29]{Maks:TA:2003}, see 2) of Remark\;\ref{rem:Maks:TA:2003-gaps}.

\S\ref{sect:examples} describes examples of vector fields such that $\imShA=\EidAFlow{k}$ for some $k\geq0$.
In particular, we recall certain ``extension'' properties $\EP{\dd}$, $(\dd\geq0)$, for singular points of $\AFld$ introduced in\;\cite{Maks:TA:2003}.
Property $\EP{0}$ guarantees $\imShA=\EidAFlow{0}$.
On the other hand we also give examples of vector fields for which $\imShA=\EidAFlow{1}\not=\EidAFlow{0}$.
At the end of\;\S\ref{sect:examples} we formulate Theorem\;\ref{th:LT_RH} which describes a class of flows satisfying $\imShA=\EidAFlow{1}$.

Further in \S\ref{sect:ES_dk_properties} we introduce another properties $\ESDL{\ld}{\dd}{k}$ which guarantee $\imShA=\EidAFlow{k}$ and prove Theorem\;\ref{th:LT_RH}.

\section{Preliminaries}\label{sect:preliminaries}
\subsection{Real Jordan normal form of a matrix}
For $k\in\NNN$ let $E_k$ be the unit $(k\times k)$-matrix, $A$ be a square $(k\times k)$-matrix, and $a,b\in\RRR$.
Define the following matrices:
$$
\Jord_p(A) = 
\left(
\begin{smallmatrix}
A      & 0      & \cdots & 0      \\
E_k      & A      & \cdots & 0    \\
\cdots & \cdots & \cdots & \cdots \\
0      & \cdots & E_k      & A
\end{smallmatrix}
\right),
\quad
R(a,b) = 
\left(\begin{smallmatrix}
 a & b \\ - b & a.
\end{smallmatrix}\right),
\quad
\Jord_{p}(a\pm ib) =  \Jord_{p}(R(a,b)).
$$
For square matrices $A,B$ it is also convenient to put
$A\oplus B = \left(\begin{smallmatrix}  A & 0 \\ 0 & B \end{smallmatrix}\right)$.

Now let $C$ be an $(n\times n)$ matrix over $\RRR$.
Then by the real variant of Jordan's normal form theorem, e.g.\;\cite{Palis_deMelo}, $C$ is similar to a matrix of the following form:
$$
\mathop\oplus\limits_{\sigma=1}^{s} \Jord_{q_{\sigma}}(a_{\sigma} \pm i b_{\sigma}))
\ \ \oplus \ \ 
\mathop\oplus\limits_{\tau=1}^{r} \Jord_{p_{\tau}}(\lambda_\tau),
$$
where $a_{\sigma}\pm i b_{\sigma} \in\CCC$ and $\lambda_{\tau} \in\RRR$ are all the eigen values of $C$.

\subsection{Deformations}\label{sect:deformations}
Let $\CntParMan$ be a topological space and $\Vman,\SmParMan,\DestMan$ be smooth manifolds.
We introduce here special types\index{} of deformations which will be used throughout the paper.

Recall that for every $k=0,\ldots,\infty$ the space $\Cr{k}{\Vman}{\DestMan}$ can be endowed with the so-called \myemph{weak topology} which we will denote by $\Wr{k}$, see\;\cite{Hirsch:DiffTop} for details.
A topology $\Wr{0}$ coincides with the compact open one, \cite{Kuratowski:Top2:1968}, while topologies $\Wr{k}$, $(k\geq1)$, can be defined as follows.

Let $J^{k}(\Vman,\DestMan)$ be the manifold of $k$-jets of $\Cont{k}$ maps $\dif:\Vman\to\DestMan$.
Associating to every such $\dif$ its $k$-jet prolongation $j^{k}\dif:\Vman\to J^{k}(\Vman,\DestMan)$, we obtain a canonical embedding $\Cr{k}{\Vman}{\DestMan} \subset \Cr{k}{\Vman}{J^{r}(\Vman,\DestMan)}$.
Endow $\Cr{k}{\Vman}{J^{k}(\Vman,\DestMan)}$ with the topology $\Wr{0}$.
Then the induced topology on $\Cr{k}{\Vman}{\DestMan}$ is called $\Wr{k}$.
Finally, the topology $\Wr{\infty}$ is generated by all $\Wr{k}$ for $k<\infty$.

For a subset $\Subsp\subset\Cr{k}{\Vman}{\DestMan}$ denote by $\Cr{k}{\SmParMan}{\Subsp}$ the space of all $\Cont{k}$ maps $\Omega:\Vman\times\SmParMan\to\DestMan$ such that $\Omega_{\sx}=\Omega(\cdot,\sx):\Vman\to\DestMan$ belongs to $\Subsp$ for each $\sx\in\SmParMan$.
Thus $\Cr{k}{\SmParMan}{\Subsp} \subset \Cr{k}{\Vman\times\SmParMan}{\DestMan}$.

\begin{definition}\label{defn:deformation}
Let $\Subsp\subset\Ci{\Vman}{\DestMan}$ be a subset.
A continuous map
\begin{equation}\label{equ:deformation}
\Omega:\Vman\times\SmParMan\times\CntParMan\to\DestMan
\end{equation}
will be called an \myemph{$\dfrm{\SmParMan}{\CntParMan}{k}$-deformation in $\Subsp$} if
\begin{enumerate}
 \item[\rm(1)]
for every $(\sx,\cx)\in\SmParMan\times\CntParMan$ the map $\Omega_{(\sx,\cx)}=\Omega(\cdot,\sx,\cx):\Vman\to\DestMan$ belongs to $\Subsp$;
 \item[\rm(2)]
for every $\cx\in\CntParMan$ the map $\Omega_{\cx}=\Omega(\cdot,\cdot,\cx):\Vman\times\SmParMan\to\DestMan$ is $\Cinf$  and the induced mapping
$$
(\Vman\times\SmParMan)\times\CntParMan \to J^{k}(\Vman\times\SmParMan,\DestMan), \qquad (x,\sx,\cx) \mapsto j^{k}\Omega_{\cx}(x,\sx)
$$
is continuous.
\end{enumerate}
\end{definition}
Roughly speaking $\SmParMan$ is the space of $\Cinf$-parameters and $\CntParMan$ is the space of ``almost'' $\Cont{k}$-parameters of the deformation $\Omega$.
We will usually denote the ``space of parameters'' $\SmParMan\times\CntParMan$ by $\ParMan$.

Evidently, $\Omega$ can be regarded as the following map
$$
\omega:\CntParMan \to \Ci{\SmParMan}{\Subsp} \subset \Ci{\Vman\times\SmParMan}{\DestMan}, \quad
\omega(\cx)(\sx)(x)=\Omega(x,\sx,\cx).
$$

Let $\SmParMan=*$ be a point, so $\Omega:\Vman\times\CntParMan\to\DestMan$ and $\omega:\CntParMan\to\Subsp$.
Suppose also that $\CntParMan$ is locally compact.
Then it is well known, e.g.\;\cite{Kuratowski:Top2:1968}, that continuity of $\Omega$ is equivalent to continuity of $\omega$ into the topology $\Wr{0}$ of $\Subsp$.
Therefore it follows from the above description of $\Wr{k}$ that a $\dfrm{*}{\CntParMan}{k}$-deformation is the same that a continuous map $\CntParMan\to\Subsp$ into the topology $\Wr{k}$.

A $\dfrm{*}{\CntParMan}{k}$-deformation will also be called a \myemph{$\dfrmc{\CntParMan}{k}$-deformation}.
Moreover, if $\CntParMan=I$, then an $\dfrmc{I}{k}$-deformation $\Omega:\Vman\times I \to \DestMan$ will be called a \myemph{$k$-homotopy}.
It can be thought as a continuous path $I\to\Subsp$ into the topology $\Wr{k}$.
In particular, a $0$-homotopy is a usual homotopy.

If $\CntParMan=*$ is a point, then $\Omega:\Vman\times\SmParMan\to\DestMan$ is a $\Cinf$ map which, by definition, belongs to $\Ci{\SmParMan}{\Subsp}$.
In this case $k$ does not matter and we will call $\Omega$ an \myemph{$(\SmParMan,\Cinf)$-deformation}.

The following lemma is left for the reader.
\begin{lemma}\label{lm:change_deform}
Let $\Omega$ be an $\dfrm{\SmParMan'}{\CntParMan'}{k}$-deformation given by\;\eqref{equ:deformation}.
Let also $\SmParMan'$ and $\DestMan'$ be smooth manifolds, $\CntParMan'$ be another topological space, $\mu:\SmParMan'\to\SmParMan$ and $u:\DestMan\to\DestMan'$ be smooth maps, and $\rho:\CntParMan'\to\CntParMan$ be a continuous map.
Then the following map
$$
\Omega:\Vman\times\SmParMan'\times\CntParMan'\to\DestMan',
\qquad
\Omega(x,\sx,\cx) = u \circ \Omega(x,\mu(\sx), \rho(\cx))
$$
is an $\dfrm{\SmParMan'}{\CntParMan'}{k}$-deformation.
\hfill\qed
\end{lemma}

\section{Shift map}\label{sect:shift-map}
\subsection{Local flows}
Let $\AFld$ be a vector field on $\Mman$ tangent to $\partial\Mman$.
Then the \myemph{orbit} of a point $x\in\Mman$ is a unique map $\AFlow_x: \RRR\supset(a_x,b_x) \to \Mman$ such that $\AFlow_x(0)=x$ and $\frac{d}{dt}\AFlow_x = \AFld(\AFlow_x)$, where $(a_x,b_x) \subset\RRR$ is the maximal interval on which a map with the previous two properties can be defined.
If $x$ is either zero or periodic point for $\AFld$, then $(a_x,b_x)=\RRR$.

By standard theorems in ODE the following subset of $\Mman\times\RRR$ 
$$
\domA = \mathop\cup\limits_{x\in\Mman} x \times (a_x, b_x),
$$
is an open, connected neighbourhood of $\Mman\times0$ in $\Mman\times\RRR$.

Then the \myemph{local flow\/} of $\AFld$ is the following map, being in fact $\Cinf$,
$$\AFlow: \Mman\times\RRR \; \supset \;\domA\longrightarrow\Mman,
\qquad
\AFlow(x,t) = \AFlow_x(t).
$$
If $\AFld$ has compact support, then $\domA=\Mman\times\RRR$, so $\AFlow$ is a \myemph{global} flow, e.g.\;\cite{Palis_deMelo}.

It is well-known that if $\AFlow$ is not global, then we can find a smooth strictly positive function $\mu:\Mman\to(0,+\infty)$ such that the flow $\BFlow$ of $\BFld=\mu\AFld$ is global, e.g.\;\cite[Cor.\;2]{Hart:Top:1983}.

\subsection{Shift map}
For open $\Vman\subset\Mman$ let $\funcAV$ be the subset of $C^{\infty}(\Vman,\RRR)$ consisting of functions $\afunc$ whose graph $\{(x,\afunc(x)) \, | \, x\in\Vman\}$ is contained in $\domA$.
It $\AFld$ either has no non-periodic orbits, or generates a global flow, then $\funcAV=\Ci{\Vman}{\RRR}$.

Then we can define the following map 
$$\ShAV: C^{\infty}(\Vman,\RRR) \;\; \supset  \;\; \funcAV \;\longrightarrow\; C^{\infty}(\Vman,\Mman)$$ 
by $\ShAV(\afunc)(x) = \AFlow(x,\afunc(x))$.
We will call $\ShAV$ the \myemph{shift map along the orbits of $\AFld$} and denote its image in $\Ci{\Mman}{\Mman}$ by $\imShAV$.

Denote by $\EAV$ the subset of $\Ci{\Vman}{\Mman}$ consisting of all mappings $\dif:\Vman\to\Mman$ such that 
\begin{enumerate}
\item
$\dif(\orb\cap\Vman) \subset\orb$ for every orbit $\orb$ of $\AFld$, and
 \item 
$\dif$ is a local diffeomorphism at every singular point $z\in\FixF\cap\Vman$.
\end{enumerate}
Let also $\EidAV{k}$, ($0\leq k\leq \infty$), be the path-component of the identity inclusion  $i_{\Vman}:\Vman\subset\Mman$ in $\EAV$  with respect to the topology $\Wr{k}$.
It consists of all $\dif\in\EAV$ which are $k$-homotopic to $i_{\Vman}$ in $\EAV$.
If $\Vman=\Mman$, then we will omit $\Vman$ and simply write $\EAFlow\defeq\EAM$, $\funcAM\defeq\funcA$, $\imShA\defeq\imShAM$, and so on.

It can be shown similarly to Lemma\;\ref{lm:imShA_EidAV} that 
$$
\imShAV \subset \EidAV{\infty}  \subset \cdots \subset \EidAV{1} \subset  \EidAV{0}.
$$
We study the problem whether $\imShAV=\EidAV{k}$ for some $k$.

Lemma\;\ref{lm:Sh-EA-under-reparametrization} below implies that $\EAV$ and $\imShAV$ do not change under reparametrizations, that is when we replace $\AFld$ with $\mu\AFld$ for some strictly positive, $\Cinf$ function $\mu:\Mman\to(0,+\infty)$.
Therefore we can always assume that $\AFld$ generates a global flow.
This simplifies many arguments.

\subsection{The kernel of shift map}
The following subset\footnote{In\;\cite{Maks:TA:2003} $\ZidAV$ was denoted by $Z_{\id}$.} $$\ZidAV\defeq\ShAV^{-1}(i_{\Vman})$$ of $\Ci{\Vman}{\RRR}$ will be called the \myemph{kernel} of $\ShAV$.
Thus $\ZidAV$ consists of all $\mu\in\Ci{\Vman}{\RRR}$ such that $\AFlow(x,\mu(x))\equiv x$ for every $x\in\Vman$.

\begin{lemma}{\rm\cite[Lm.\;5\;\&\;7]{Maks:TA:2003}}\label{lm:ker1}
Let $\afunc,\bfunc\in\funcAV$.
Then $\ShAV(\afunc)=\ShAV(\bfunc)$ iff $\afunc-\bfunc\in\ZidAV$.
Hence if $\funcA=\Ci{\Vman}{\RRR}$, then $\ZidAV$ is a group with respect to the point-wise addition and $\ShAV$ yields a bijection between the factor group $\Ci{\Vman}{\RRR}/\ZidAV$ and the image $\imShAV$.

Every $\theta\in\ZidAV$ is locally constant on orbits of $\AFld$.
If $x\in\Vman$ is non-periodic, then $\theta(x)=0$.
If $x$ is periodic, then $\theta(x) =n\Per(x)$  for some $n\in\ZZZ$.
\end{lemma}

\begin{theorem}[Description of $\ZidAV$]\label{th:ker2}
{\rm\cite[Th.\;12]{Maks:TA:2003}, \cite{Maks:period_functions}}
Let $\Vman \subset\Mman$ be a \myemph{connected}, open subset.
If $\Int{\FixF}\cap\Vman\not=\varnothing$, then 
$$\ZidAV=\{\mu\in\funcAV \ : \ \mu|_{\Vman\setminus\Int{\FixF}}=0\}.$$
Suppose $\FixF\cap\Vman$ is nowhere dense in $\Vman$.
Then one of the following possibilities for $\ZidAV$ is realized:

\myemph{Non-periodic case:} $\ZidAV=\{0\}$, so $\ShAV$ is injective.

\myemph{Periodic case:} $\ZidAV=\{n\,\theta\}_{n\in\ZZZ}$ for some $\theta\in\Ci{\Vman}{\RRR}$ called the \myemph{positive generator} of $\ZidAV$ and having the following properties:
\begin{enumerate}
\item $\theta>0$ on all of $\Vman$, so $\Vman\setminus\FixF$ consists of periodic points only, and therefore $\funcAV=\Ci{\Vman}{\RRR}$;
\item there exists an open and everywhere dense subset $\Qman\subset\Vman$ such that $\theta(x)=\Per(x)$ for all $x\in\Qman$;
\item for every orbit $\orb$ of $\AFld$ the restriction $\theta|_{\orb\cap\Vman}$ is constant;
\item $\theta$ extends to a $\Cinf$ function on the $\AFld$-invariant set $\Uman=\AFlow(\Vman\times\RRR)$ and the vector field $\BFld=\theta\AFld$ generates a circle action on $\Uman$:
$$\BFlow:\Uman\times S^1\to\Uman, \qquad \BFlow(x,t) = \AFlow(x,\theta(x)\cdot t),$$ where $x\in\Uman$, $t\in S^1=\RRR/\ZZZ$.
\end{enumerate}
\end{theorem}

\subsection{Reparametrization of time}
The following lemma describes the behavior of $\imShA$ and $\EAFlow$ under reparametrizations of time.

\begin{lemma}
\label{lm:Sh-EA-under-reparametrization}{\rm c.f.\cite{Maks:reparam-sh-map}}
Let $\mu:\Mman\to\RRR$ be a $\Cinf$ function.
Consider the vector field $\BFld=\mu\AFld$.
Denote by $\FixA$ and $\FixB$ the sets of singular points of $\AFld$ and $\BFld$ respectively.
Then the following statements hold true.

{\rm1)} $\EBFlow\subset\EAFlow$, and $\EBFlow=\EAFlow$ iff $\mu\not=0$ on $\Mman\setminus\FixA$.

{\rm2)}
$\imSh{\BFlow} \subset \imSh{\AFlow}$, and $\imSh{\BFlow}=\imSh{\AFlow}$ iff $\mu\not=0$ on $\Mman\setminus\Int{\FixA}$.
\end{lemma}
\begin{proof}
1) Evidently, every orbit of $\BFld$ is contained in some orbit of $\AFld$, whence $\EBFlow\subset\EAFlow$.
Suppose $\mu\not=0$ on $\Mman\setminus\FixA$.
Then the foliations by orbits of $\AFld$ and $\BFld$ coincide, whence $\EBFlow=\EAFlow$.
Conversely, if $\mu(x)=0$ for some $x\in\Mman\setminus\FixA$, then $x$ is a fixed point for $\BFlow$, so $\dif(x)=x$ for every $\dif\in\EBFlow$, while $\EAFlow$ contains maps $g$ such that $g(x)\not=x$.
Hence $\EBFlow\subsetneq\EAFlow$.

2) Define the following function $\afunc:\domB\to\RRR$ by
\begin{equation}\label{equ:reparam-function}
\gamma(x,s) = \int_{0}^{s}\mu(\BFlow(x,t))dt.
\end{equation}
Then it is well known that $\BFlow(x,t)=\AFlow(x,\gamma(x,t))$, e.g.\;\cite[Prop.\;1.28]{ChiconeODE:2006}.
It follows that 
\begin{equation}\label{equ:G_xax__F_x_gxax}
\BFlow(x,\afunc(x))=\AFlow\bigl(x,\gamma(x,\afunc(x))\bigr), \qquad \afunc\in\funcB.
\end{equation}
Hence $\imSh{\mu\AFld} \subset \imShA.$

Suppose $\mu\not=0$ on $\Mman\setminus\Int{\FixA}$.
Then $\mu\not=0$ on some neighbourhood of $\Mman\setminus\Int{\FixA}$, so we can find another function $\nu:\Mman\to(0,+\infty)$ which is strictly positive and coincides with $\mu$ on a neighbourhood of $\Mman\setminus\Int{\FixA}$.
Hence $\BFld=\mu\AFld=\nu\AFld$ and $\AFld=\tfrac{1}{\nu}\BFld$.
Therefore $\imSh{\AFlow} \subset \imSh{\BFlow}$, and thus $\imSh{\BFlow}=\imSh{\AFlow}$.

Conversely, suppose $\mu(x)=0$ for some $x\in\Mman\setminus\Int{\FixA}$.
Let $\bfunc\in\funcA$ be a function such that $\bfunc(x)\not=0$, and $g=\ShA(\bfunc)$.
If $x$ is periodic, we will assume that $0<\bfunc(x)<\Per(x)$.
We claim that $g\in\imShA\setminus\imSh{\BFlow}$.

Suppose $g(x)=\BFlow(x,\afunc(x))$ for some $\afunc\in\funcB$.
Put $\bfunc'(x) = \gamma(x,\afunc(x))$.
Then by\;\eqref{equ:G_xax__F_x_gxax} $g(x) =\AFlow(x,\bfunc'(x))$, whence $\nu=\bfunc-\bfunc'\in\ZidA$.
On the other hand by\;\eqref{equ:reparam-function} $\gamma(x,s)=0$ for all $s$.
In particular, $\bfunc'(x)=0$.
We will show that $\bfunc(x)=0$ which contradicts to the assumption.
Consider three cases.

(a) $\Int{\FixA}\not=\varnothing$.
Then by Theorem\;\ref{th:ker2} $\bfunc-\bfunc'$ vanishes on $\Mman\setminus\Int{\FixF}$.
In particular, $\bfunc(x)=0$.

(b) $\Int{\FixA}=\varnothing$ and $\ZidA=\{0\}$.
Then $\bfunc\equiv\bfunc'$.

(c) $\Int{\FixA}=\varnothing$ and $\ZidA=\{n\theta\}_{n\in\ZZZ}$.
Then $\bfunc-\bfunc'=n\theta$ for some $n\in\ZZZ$.
But by assumption $0<\bfunc(x)<\Per(x)\leq\theta(x)$ and $\bfunc'(x)=0$, whence $\bfunc(x)=0$ as well.
\end{proof}

\begin{corollary}\label{cor:imSh_not_Eid}
Let $\mu:\Mman\to\RRR$ be a $\Cinf$ function such that $\mu(x)=0$ for some $x\in\FixF\setminus\Int{\FixF}$, and $\BFld=\mu\AFld$.
Then $\imShB \not=\EidBFlow{\infty}$.
\end{corollary}
\begin{proof}
By Lemma\;\ref{lm:Sh-EA-under-reparametrization} $\imShB\subsetneq\imShA\subseteq\EidAFlow{\infty}=\EidBFlow{\infty}$.
\end{proof}

\subsection{Property $\GSF$}\label{sect:prop-GSF}
In\;\cite{Maks:CEJM:2009} shift map was used to describe a class of vector fields $\AFld$ having the following property which will be called in the present paper $\GSF$%
\footnote{In \cite{Maks:CEJM:2009} I used the term \myemph{parameter rigidity} for this property.
But, as the referee of the present paper noted, usually the action of a group $\Gamma$ is called \myemph{parameter rigid} if any other $\Gamma$-action with the same orbits is smoothly conjugate to the original one (up to a linear reparametrization of orbits). Therefore here another term $\GSF$ is used. }: 
\begin{definition}\label{defn:fol-cond}{\rm(c.f.\;\cite{Santos:ETDS:2007}, \cite{Maks:CEJM:2009})}
Say that a $\Cinf$ vector field $\AFld$ on a manifold $\Mman$ has property $\GSF$ if for any $\Cinf$ vector field $\BFld$ on $\Mman$ such that every orbit of $\BFld$ is contained in some orbit of $\AFld$ there exists a $\Cinf$ function $\afunc$ such that $\BFld=\afunc \AFld$.
\end{definition}
Evidently, $\afunc$ always exists on the set of non-singular points of $\AFld$ and the problem is to prove that it can be smoothly extended to all of $\Mman$.
In particular, every non-singular vector field has $\GSF$.

It follows from\;\cite{Maks:CEJM:2009} that if $\imShA=\EidAFlow{\infty}$ and the map $\ShA$ of $\AFld$ satisfies a \myemph{smooth variant of covering path axiom at each $z\in\FixF$}, then $\AFld$ has property $\GSF$.
As an application of Theorem\;\ref{th:suffcond_for_imShA_EidAk} we will present a class of vector fields having property $\GSF$, see Theorem~\ref{th:LT_RH}. 

To explain the notation $\GSF$ let us reformulate this definition in algebraic terms.
Notice that the space $\VFM$ of $\Cinf$ vector fields on $\Mman$ can be regarded as a $\CMR$-module.
For each $\AFld\in\VFM$ define the \myemph{principal submodule} $\prs{\AFld}$ of $\AFld$ as follows:
$$
\prs{\AFld} := \{ \afunc\AFld \ : \ \afunc\in\CMR \}.
$$
Thus the inclusion $\prs{\BFld}\subset\prs{\AFld}$ means that $\BFld=\afunc\AFld$ for some $\afunc\in\CMR$, so \myemph{$\BFld$ is smoothly divided by $\AFld$}.

On the other hand, we can introduce the following relation $\fleq$ on $\VFM$ being reflexive and transitive.
Let $\AFld,\BFld\in\VFM$.
We say that $\BFld\fleq\AFld$ if and only if each orbit of $\BFld$ is contained in some orbit of $\AFld$.

Evidently, $\prs{\BFld}\subset\prs{\AFld}$ implies $\BFld\fleq\AFld$.

Then $\AFld$ satisfies condition $\GSF$ iff $\BFld\fleq\AFld$ implies $\prs{\BFld}\subset\prs{\AFld}$ for any $\BFld\in\VFM$.
In other words, $\prs{\AFld}$ is the {\bf greatest principal submodule} among all principal submodules $\prs{\BFld}$ whose {\bf foliation} by orbits  is obtained by partitioning the corresponding foliation of $\AFld$.

In Theorem\;\ref{th:LT_RH} we present a class of vector fields satisfying $\GSF$.
This will also extend\;\cite[Th.\;11.1]{Maks:CEJM:2009}.

\subsection{Comparison of shift maps for open sets}
\begin{lemma}\label{lm:per-non-per}
Let $\Wman\subset\Vman$ be two open connected subsets of $\Mman$.
If $\ShAV$ is periodic, then so is $\ShAW$.
Moreover, if $\vfunc$ and $\wfunc$ are the corresponding positive generators of $\ZidAV$ and $\ZidAW$ respectively, then $\vfunc|_{\Wman} = \wfunc.$
Hence if $\ShAW$ is non-periodic, then so is $\ShAV$.
\end{lemma}
\begin{proof}
We have that $\Int\FixF\cap\Vman=\varnothing$ and $\AFlow(x,\vfunc(x))=x$ for all $x\in\Vman$.
In particular, $\Int\FixF\cap\Wman=\varnothing$ as well and $\AFlow(x,\vfunc(x))=x$ for all in $\Wman$.
Hence $\vfunc|_{\Wman}\in\ZidAW$ is a non-zero shift function for the identity inclusion $i_{\Wman}:\Wman\subset\Mman$, and therefore $\ShAW$ is also periodic.

Let $\wfunc$ be the positive generator of $\ZidAW$.
Then $\vfunc|_{\Wman} = n\wfunc$ for some $n\in\NNN$.
We claim that $n=1$.

Indeed, by (4) of Theorem\;\ref{th:ker2} one can assume that $\Wman$ and $\Vman$ are $\AFld$-invariant.
Define the following map
$$\dif:\Vman\to\Vman, \qquad \text{by} \qquad \dif(x) =\AFlow(x,\vfunc(x)/n).$$
Then it follows from (3) of Theorem\;\ref{th:ker2} that $\dif^{k}(x)=\AFlow(x,k\,\vfunc(x)/n)$ for all $k=1,\ldots,n$.
In particular, $\dif^{n}=\id_{\Vman}$ and so $\dif$ yields a $\ZZZ_{n}$-action on $\Vman$.
Evidently, this action is fixed on a \myemph{non-empty open set} $\Wman$: if $x\in\Wman$, then
$$
\AFlow(x,\vfunc(x)/n) = \AFlow(x,\wfunc(x)) = x.
$$
As $\Vman$ is \myemph{connected}, we get from the well-known theorem of M.\;New\-man that $\dif=\id_{\Vman}$, see e.g.\;\cite{Newman:QJM:1931,Dress:Topol:1969}.
Thus $\tfrac{1}{n}\vfunc\in\ZidAV$.
But $\vfunc$ is the least positive function which generates $\ZidAV$, whence $n=1$.
\end{proof}

\subsection{Function of periods}
Suppose that all points of $\AFlow$ are periodic.
Consider the function $$\Per:\Mman\to(0,+\infty)$$ associating to each $x\in\Mman$ its period $\Per(x)$ with respect to $\AFlow$.
Then by D.\;B.\;A.\;Epstein\;\cite[\S5]{Epstein:AnnMath:1972} $\Per$ is lower semicontinuous and the set $B$ of its continuity points is open, see also D.\;Montgomery\;\cite{Montgomery:AJM:1937}.

We call the shift map $\ShA$ \myemph{periodic} if, in accordance with (2) of Theorem\;\ref{th:ker2}, $B$ is everywhere dense in $\Mman$, $\Per$ is $\Cinf$ on $B$, and even extends to a $\Cinf$ function on all of $\Mman$.
Otherwise, $\ShAM$ is \myemph{non-periodic}.
In the latter case the zero function $\mufunc\equiv 0$ is a unique $\Cinf$ function on $\Mman$ satisfying $\AFlow(x,\mufunc(x))=x$ for every $x\in\Mman$.
In particular, $\Per$ can not be extended to a $\Cinf$ function on all of $\Mman$.

We will now discuss obstructions for continuity of $\Per$.

It is possible that $\Per$ is unbounded near some points on $\Mman$.
The first examples of this sort seem to be constructed by G.\;Reeb\;\cite{Reeb:ASI:1952}.
Further examples of flows with all orbits closed and with locally unbounded period function were obtained by D.\;B.\;A.\;Epstein\;\cite{Epstein:AnnMath:1972} (a \myemph{real analytic} flow on a non-compact $3$-manifold),  D.\;Sullivan\;\cite{Sullivan:BAMS:1976, Sullivan:PMIHES:1976} (a $\Cinf$ flow on a \myemph{compact} $5$-manifold $S^3\times S^1\times S^1$), D.\;B.\;A.\;Epstein and E.\;Vogt\;\cite{EpsteinVogt:AnnMath:1978} (a flow on a compact $4$-manifold defined by \myemph{polynomial equations, with the vector field defining the flow given by polynomials}), E.\;Vogt\;\cite{Vogt:ManuscrMath:1977}, and others.

On the other hand, $\Per$ continuously extends from $B$ to all of $\Mman$ for the case of suspension flows (D.\;Montgomery\;\cite{Montgomery:AJM:1937}) and if $\Mman$ is a compact orientable $3$-manifold (D.\;B.\;A.\;Epstein\;\cite{Epstein:AnnMath:1972}).
In these cases $\AFlow$ can be reparametrized to a circle action.
More general sufficient conditions for existence of such reparametrizations were obtained by R.\;Edwards, K.\;Millett, and D.\;Sullivan\;\cite{EdwardsKennethSullivan:Top:1977}.

It should also be noted that due to A.\;W.\;Wadsley\;\cite{Wadsley:JDG:1975} an existence of a circle action with the orbits of $\AFld$ is equivalent to the existence a Riemannian metric on $\Mman$ in which all the orbits are geodesic.
Moreover, if $\BFlow:\Mman\times S^1\to\Mman$ is a smooth circle action, then by R.\;Palais\;\cite[Th.\;4.3.1]{Palais:AnnM:1961} $\Mman$ has an invariant Riemannian metric and by M.\;Kankaanrinta\;\cite{Kankaanrinta:TA:2005} this metric can be made complete.
Also if $\Mman$ is compact, then due to G.\;D.\;Mostow\;\cite{Mostow:AnnM:1957} and R.\;Palais\;\cite{Palais:JMM:1957} this action can be made orthogonal with respect to some embedding of $\Mman$ into a certain finite-dimensional Euclidean space.

\subsection{$\PN$-points}
Let $\Vman\subset\Mman$ be an open connected subset such that $\VmanS:=\Vman\setminus\FixF$ is also connected.
Suppose that the shift map $\ShAVS$ is periodic and let $\theta:\VmanS\to(0,+\infty)$ be the positive generator of $\ZidAVS$.
Then due to Lemma\;\ref{lm:per-non-per} the shift map $\ShAV$ is also periodic if and only if $\theta$ extends to a $\Cinf$ strictly positive function on all of $\Vman$.

Again one of the reasons for $\ShAV$ to be non-periodic is unboundedness of $\theta$ at some points of $\FixF$.
These effects are reflected in the following definition.

\begin{definition}\label{defn:PN-point}
Say that $z\in\FixF$ is a \myemph{$\PN$-point} if there exists an open neighbourhood $\Vman$ of $z$ such that $\VmanS=\Vman\setminus\FixF$ is connected and the shift map $\ShAVS$ is periodic, while for any open connected neighbourhood $\Wman\subset\Vman$ of $z$ the shift map $\ShAW$ is non-periodic.

Let $\theta:\VmanS\to(0,+\infty)$ be the positive generator of $\ZidAVS$.
Then $z$ will be called a \myemph{strong $\PN$-point}, if $\lim\limits_{x\to z}\theta(x)=+\infty$.
\end{definition}

The following lemma is a consequence of results obtained in\;\cite{Maks:period_functions}.
\begin{lemma}\label{lm:suff-cond-for-PN-point}{\rm c.f.\;\cite{Maks:period_functions}}
Let $z\in\FixF$.
Suppose that there exists an open connected neighbourhood $\Vman$ of $z$ such that the shift map $\ShAVS$ is periodic, where $\VmanS=\Vman\setminus\FixF$.
Let $B$ be the real Jordan normal form of the linear part $j^1\AFld(z)$ of $\AFld$ at $z$.
Then the following statements hold true.

\begin{enumerate}
\item 
$\Re(\lambda)=0$ for every eigen value $\lambda$ of $B$, so $B$ is similar to
$$
\mathop\oplus\limits_{\sigma=1}^{s} \Jord_{q_{\sigma}}(\pm i b_{\sigma}) \ \ \oplus \ \
\mathop\oplus\limits_{\tau=1}^{r} \Jord_{p_{\tau}}(0),
$$
for some $q_{\sigma},p_{\tau}\in\NNN$ and $b_{\sigma}\in\RRR\setminus\{0\}$.

\item
If $\ShAV$ is periodic, then $B\not=0$ and is similar to
$$
\mathop\oplus\limits_{\sigma=1}^{s} \Jord_{1}(\pm i b_{\sigma})
\ \oplus \  \mathop\oplus\limits_{\tau=1}^{r}\Jord_{1}(0) =
\left(\begin{smallmatrix} 
0 & b_1 \\ -b_1 & 0  \\
 &  & \cdots \\
 &  &    & 0 & b_s \\
 &  &    & -b_s & 0 \\
 & & & & & 0 \\
 & & & & & & \cdots 
\end{smallmatrix}\right),
$$
for some $s\geq1$ and $b_1,\ldots,b_s\in\RRR\setminus0$.

\item
If $B=0$, or if $B$ contains either a block $\Jord_{q}(\pm ib)$ or $\Jord_{q}(0)$ with $q\geq2$, then $z$ is a $\PN$-point.
Moreover, in this case $\theta$ is unbounded at $z$.
\end{enumerate}
\end{lemma}
\begin{proof}
(1)
Since $\ShAVS$ is periodic, we have that $\VmanS$ consists of periodic points only.
If $B$ has an eigen value $\lambda$ with $\Re(\lambda)\not=0$, then by the Hadamard-Perron's theorem, e.g.\;\cite{HirschPughShub:LNM538:1977}, there exists a non-periodic orbit $\orb$ of $\AFld$ such that $z\in\cl{\orb}\setminus\orb$.
Hence $\VmanS\cap\orb\not=\varnothing$.
This is a contradiction.

(3) It is shown in\;\cite{Maks:period_functions} that if $B=0$ or if $B$ contains either a block $\Jord_{q}(\pm ib)$ or $\Jord_{q}(0)$ with $q\geq2$, then there exists a sequence $\{x_i\}_{i\in\NNN}$ in $\VmanS$ converging to $z$ such that $\lim\limits_{i\to\infty}\Per(x_i)=+\infty$.
Since $\theta(x_i)=n_i\Per(x_i)$ for some $n_i\in\NNN$, we obtain $\lim\limits_{i\to\infty}\theta(x_i)=+\infty$ as well.

(2)
Suppose $\ShAV$ is periodic, so we have a circle action $\BFlow$ on $\AFlow(\Vman\times\RRR)$ defined in (4) of Theorem\;\ref{th:ker2}.
This action induces a linear circle action on the tangent space $T_{z}\Vman$, whence it follows from standard results about representations of $SO(2)$ that $B$ is similar to the matrix 
$\oplus_{\sigma=1}^{s} \Jord_{1}(\pm i b_{\sigma}) \oplus_{\tau=1}^{r}\Jord_{1}(0)$.
Notice that these arguments do not prove that $B\not=0$, however this holds due to (3).
\end{proof}

Statement (3) of Lemma\;\ref{lm:suff-cond-for-PN-point} does not claim that $z$ is a strong $\PN$-point, i.e. $\lim\limits_{i\to\infty}\theta(x_i)=+\infty$ for any sequence of periodic points $\{x_i\}_{i\in\NNN}$ converging to $z$.

\begin{example}\label{exmp:top_sing}\rm
Let $a\leq b\in\NNN$ and $f:\RRR^2\to\RRR$ be the polynomial defined by $f(x,y)=x^{2a}+y^{2b}$.
Then the Hamiltonian vector field 
$$
\AFld(x,y) = -f'_{y} \tfrac{\partial}{\partial x} + f'_{x} \tfrac{\partial}{\partial y}=
-2b y^{2b-1} \tfrac{\partial}{\partial x} + 2 a x^{2a-1} \tfrac{\partial}{\partial y}
$$
of $f$ has the following property: the origin $\orig\in\RRR^2$ is a unique singular point of $\AFld$, and all other orbits are concentric closed curves wrapped once around $\orig$, see Figure\;\ref{fig:top_center}.
It follows from smoothness Poincar\'e's return map for orbits of $\AFld$ that the period function $\theta:\RRR^2\setminus\orig\to(0,+\infty)$ defined by $\theta(z)=\Per(z)$ is $\Cinf$.
Also notice that, $j^1\AFld(\orig)$ is given by one of the following matrices:
$$
\begin{array}{ccccc}
\left(\begin{smallmatrix} 0 &  2 \\ -2 & 0  \end{smallmatrix}\right)
& \quad\quad  & 
\left(\begin{smallmatrix} 0 &  2 \\  0 & 0  \end{smallmatrix}\right)
& \quad\quad  &
\left(\begin{smallmatrix} 0 &  0 \\ 0 & 0  \end{smallmatrix}\right) \\ [2mm]
1)~a=b=1,
&  &
2)~a=1, b\geq2,
& &
3)~a, b\geq2.
\end{array}
$$

In the case 1) $\AFld$ is linear and its flows is given by $\AFlow(z,t)=e^{2 i t}z$.
Hence $\theta(z) = \Per(z)=\pi$ for all $z\in\RRR^2\setminus\orig$, so $\theta$ extends to a $\Cinf$ function on all of $\RRR^2$ if we put $\theta(\orig)=\pi$.

In the case 2) $j^1\AFld(\orig)$ is nilpotent and in the case 3) $j^1\AFld(\orig)=0$.
Then by (3) of Lemma\;\ref{lm:suff-cond-for-PN-point} $\theta$ is unbounded at $\orig$, so $\orig$ is a $\PN$-point.
Moreover, it easily follows from the structure of orbits of $\AFld$ that in fact
$\lim\limits_{x\to\orig}\theta(z)=+\infty$, i.e. $\orig$ is a strong $\PN$-point for $\AFld$.
\end{example}
\begin{figure}[ht]
\begin{tabular}{ccccc}
\includegraphics[height=1.2cm]{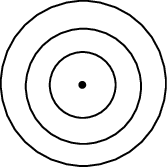} & \qquad\qquad &
\includegraphics[height=1.2cm]{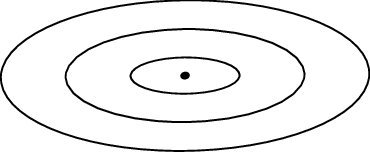} & \qquad\qquad &
\includegraphics[height=1.2cm]{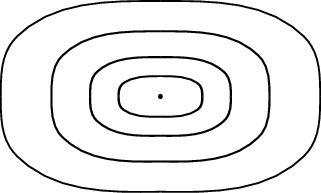} \\
1) & & 2) & & 3)
\end{tabular}
\caption{}\protect\label{fig:top_center}
\end{figure}

\section{Shift functions on the set of regular points}\label{sect:shift-func-at-reg}
Throughout the paper we will assume that $\SmParMan$ is a connected smooth manifold, $\CntParMan$ is a \myemph{path connected} and \myemph{locally path connected} topological space, $\sx_0\in\SmParMan$, and $\tau_0\in\CntParMan$.

Let $\Vman\subset\Mman$ be an open subset and $\Omega:\Vman\times\SmParMan\times\CntParMan\to\Mman$ be an $\dfrm{\SmParMan}{\CntParMan}{k}$-deformation in $\EAV$, and  $A\subset\Vman\times\SmParMan\times\CntParMan$ be a subset.
Then a function $\Lambda:A\to\RRR$ will be called a \myemph{shift function} for $\Omega$ if 
$$
\Omega(x,\sx,\cx)=\AFlow(x,\Lambda(x,\sx,\cx)), \qquad 
\forall (x,\sx,\cx)\in A.
$$

\begin{theorem}\label{th:shift-function-for-deformations}{\rm c.f.\;\cite[Th.\;25]{Maks:TA:2003}.}
Suppose $\Vman\cap\FixF=\varnothing$.

{\rm(1)}
Let $(x_0,\sx_0,\cx_0)\in\Vman\times\SmParMan\times\CntParMan$ and $a\in\RRR$ be such that \begin{equation}\label{equ:Omega_x0__F_x0_a}
\Omega(x_0,\sx_0,\cx_0)=\AFlow(x,a).
\end{equation}
Then there exist connected neighbourhoods $\Wman_{x_0}\subset\Vman$ of $x$, $\Wman_{\sx_0}\subset\SmParMan$ of $\sx_0$, $\Wman_{\cx_0}\subset\CntParMan$ of $\cx_0$, and a unique continuous shift function 
$$
\Delta:\Wman_{x_0}\times\Wman_{\sx_0}\times\Wman_{\cx_0} \to \RRR
$$
for $\Omega$ such that $\Delta(x_0,\sx_0,\cx_0)=a$.
Thus $\Omega(x,\sx,\cx)=\AFlow(x,\Delta(x,\sx,\cx))$ for all $(x,\sx,\cx)\in\Wman_{x_0}\times\Wman_{\sx_0}\times\Wman_{\cx_0}$.
Moreover, $\Delta$ is a $\dfrm{\Wman_{\sx_0}}{\Wman_{\cx_0}}{k}$-deformation in $\Ci{\Wman_{x_0}}{\RRR}$.

{\rm(2)}
Any continuous shift function $\Lambda:\Vman\times\SmParMan\times\CntParMan\to\RRR$ for $\Omega$ is an $\dfrm{\SmParMan}{\CntParMan}{k}$-deformation. 

{\rm(3)}
Denote $\ParMan=\SmParMan\times\CntParMan$.
Suppose that 
\begin{enumerate}
 \item[\rm(a)]
for each $x\in\Vman$ the map 
$\Omega_{x}:\ParMan\to\orb_x$ defined by $\Omega_{x}(\pp) = \Omega(x,\pp)$ is null-homotopic, (this holds e.g. when $\ParMan$ is simply connected i.e. $\pi_1\ParMan=0$, or when $\AFld$ has no closed orbits),
 \item[\rm(b)]
for some $\pp_0\in\ParMan$ the map $\Omega_{\pp_0}$ has a continuous shift function $\afunc:\Vman\to\RRR$, i.e. $\Omega_{\pp}(x)=\AFlow(x,\afunc(x))$.
\end{enumerate}
Then there exists a unique continuous shift function $\Lambda:\Vman\times\ParMan\to\RRR$ for $\Omega$ such that $\Lambda_{\pp_0}=\afunc$.
\end{theorem}
\begin{proof}
Let $x_0\in\Vman$.
By assumption $\Vman\cap\FixF=\varnothing$, so $x_0$ is a regular point of $\AFld$.
Hence there exist $\eps>0$, a neighbourhood $\Uman$ of $x_0$, and a diffeomorphism
$\eta:\Uman\to\RRR^{n-1}\times(-\eps,\eps)$
such that in the coordinates $(y,s)$ on $\Uman$ induced by $\eta$ we have that 
\begin{equation}\label{equ:flow-at-flow-box}
\AFlow((y,s),t)=(y,s+t), \qquad \text{whenever} \ s,s+t\in(-\eps,\eps).
\end{equation}
We will call $\Uman$ a $\eps$-flow-box at $x_0$.
Notice that for any periodic point $z\in\Uman$ its period $\Per(z)\geq2\eps$.

\myemph{We will always assume below that for any periodic point $z\in\Uman$} 
\begin{equation}\label{equ:Perz_Aeps}
\Per(z)>10\eps.
\end{equation}
To achieve this it suffices to replace $\eps$-flow-box $\Uman$ e.g. with the ``central'' $\tfrac{\eps}{6}$-flow-box $\Uman'=\eta^{-1}\bigl(\RRR^{m-1}\times(-\tfrac{\eps}{6},\tfrac{\eps}{6})\bigr)$.
Then $\Per(z)\geq2\eps=12\cdot\tfrac{\eps}{6}> 10\cdot\tfrac{\eps}{6}$ for each periodic point $z\in\Uman'$.

Let also $p:\Uman\to(-\eps,\eps)$ be the projection to the last coordinate, i.e. $p(y,s) = s$.

\medskip 

(1) Let $\Uman$ be an $\eps$-flow box at $x_0$ satisfying\;\eqref{equ:Perz_Aeps}.
Then by\;\eqref{equ:Omega_x0__F_x0_a} $\AFlow_{-a}\circ\Omega(x_0,\sx_0,\cx_0)=x_0$, hence there exists a neighbourhood $\Wman$ of $(x_0,\sx_0,\cx_0)$ in $\Vman\times\SmParMan\times\CntParMan$ such that $\AFlow_{-a}(\Wman)\subset\Uman$.
We can also assume that $\Wman$ has the desired form $\Wman_{x_0}\times\Wman_{\sx_0}\times\Wman_{\cx_0}$ with connected multiples.

Define the function $\Delta:\Wman\to\RRR$ by
\begin{equation}\label{equ:local_formulas_for_shift_func}
\Delta(x,\sx,\cx)=p\circ\AFlow_{-a}\circ\Omega(x,\sx,\cx) - p(x) + a.
\end{equation}
Then it easily follows from~\eqref{equ:flow-at-flow-box} that $\Delta$ is a shift function for $\Omega$ satisfying statement (1).
In particular, it follows from Lemma\;\ref{lm:change_deform} that $\Delta$ is a $\dfrm{\Wman_{\sx_0}}{\Wman_{\cx_0}}{k}$-deformation as well as $\Omega$.
The following statement implies uniqueness of $\Delta$.
In fact it proves much more.

\begin{claim}\label{clm:local_shift_func_is_unique}
Let $(x,\sx,\cx)\in\Wman$ and $b\in\RRR$ be such that 
$$
\Omega(x,\sx,\cx)=\AFlow(x,b) \qquad \text{and} \qquad |\Delta(x,\sx,\cx) - b| \leq 10\eps.
$$
Then $\Delta(x,\sx,\cx)=b$.
\end{claim}
\begin{proof}
For simplicity denote $\xi=(x,\sx,\cx)$.
Notice that $\Omega(\xi)\in\orb_{x}$.
If $x$ is non-periodic, then there can exist a \myemph{unique} $c\in\RRR$ such that $\Omega(\xi)=\AFlow(x,c)$, whence $b=c=\Delta(\xi)$.

Suppose $x$ is periodic.
Then $\Delta(\xi)-b=k\cdot\Per(x)$ for some $k\in\ZZZ$.
On the other hand by\;\eqref{equ:Perz_Aeps} $\Per(\Omega(\xi))=\Per(x)>10\eps$.
Therefore
$$ |k\cdot\Per(x)| = |\Delta(\xi)-b| \leq 10\eps < \Per(x),$$
whence $k=0$.
\end{proof}
\begin{corollary}\label{cor:local_shift_func_is_unique}
Let $A\subset\Wman$ be a connected subset, and $\Delta': A\to \RRR$ be a continuous shift function for $\Omega$ such that for some $\xi=(x,\sx,\cx)\in A$ we have that $|\Delta(\xi)-\Delta'(\xi)|<10\eps$.
Then $\Delta=\Delta'$ on $A$.
\end{corollary}
\begin{proof}
Put 
$$
\widetilde{B}=\{ \eta\in A \ | \ |\Delta(\eta)-\Delta'(\eta)|<10\eps\},\quad
B=\{ \eta\in A \ | \ \Delta(\eta)=\Delta'(\eta)\}.
$$
We have to show that $B=A$.

Evidently, $B$ is closed in $A$ while $\widetilde{B}$ is open.
Moreover, by Claim\;\ref{clm:local_shift_func_is_unique} $B=\widetilde{B}$ and $\xi\in B\not=\varnothing$.
Hence $B=\widetilde{B}=A$.
\end{proof}

Statement (2) is a direct consequence of (1).

\smallskip

\newcommand\nbh[2]{\Wman^{(#1,#2)}}
\newcommand\nbhX[2]{\Wman^{(#1,#2)}_{#1}}
\newcommand\nbhP[2]{\Wman^{(#1,#2)}_{#2}}
\newcommand\fnc[2]{\Delta^{(#1,#2)}}
\newcommand\lmd[1]{\Lambda^{#1}}

\newcommand\mbh[1]{\Wman^{#1}}
\newcommand\enc[1]{\Delta^{#1}}

\newcommand\nbhxp{\nbh{x}{\pp}}
\newcommand\nbhxpX{\nbhX{x}{\pp}}
\newcommand\nbhxpP{\nbhP{x}{\pp}}
\newcommand\fncxp{\fnc{x}{\pp}}

(3)
For $x\in\Vman$ let
$$\AFlow_{x}:\RRR \longrightarrow \orb_{x} \subset\Mman,
\qquad
\AFlow_{x}(\tim)=\AFlow(x,\tim)
$$
be the map representing the orbit $\orb_{x}$ of $x$.
Since $\Omega$ is a deformation in $\EAV$, we have that $\Omega(x\times \ParMan) \subset \orb_{x}$. By assumption this map is null-homotopic.
Then by the path covering property for covering maps there exists a continuous function $\lmd{x}:x\times\ParMan \to \RRR$ such that $\lmd{x}(\pp_0)=\afunc(x)$ and the following diagram is commutative:
$$
\begin{matrix}
\xymatrix{
                                                         & \RRR \ar[d]^{\AFlow_{x}} \\
 x\times\ParMan \ar@{.>}[ru]^{\lmd{x}} \ar[r]^{~~\Omega} & \orb_{x}
}
\end{matrix} 
 \qquad
\text{i.e}
\qquad
\Omega(x,\pp) = \AFlow(x,\lmd{x}(\pp)),
$$

Define the following function $\Lambda:\Vman\times\ParMan\to\RRR$ by $\Lambda(x,\pp)=\lmd{x}(\pp)$.
Then $\Lambda(x,\pp_0)=\afunc(x)$ and $\Omega(x,\pp)=\AFlow(x,\Lambda(x,\pp))$ for $(x,\pp)\in\Vman\times\ParMan$.
It remains to show that $\Lambda$ is continuous.

By (1) for each $(x,\pp)\in\Vman\times\ParMan$ there are connected neighbourhoods $\nbhxpX\subset\Vman$ of $x$ and $\nbhxpP\subset\ParMan$ of $\pp$, and a continuous shift function 
$$\fncxp:\nbhxp=\nbhxpX\times\nbhxpP\longrightarrow\RRR$$ for $\Omega$ such that 
\begin{equation}\label{equ:fncxp=lmdxp}
\fncxp(x,\pp)=\lmd{x}(\pp).
\end{equation}

Our aim is to show that $\fnc{x}{\pp} = \fnc{x'}{\pp'}$ on $\nbh{x}{\pp}\cap\nbh{x'}{\pp'}$ for all $(x,\pp),(x',\pp')\in\Vman\times\ParMan$.
Hence the functions $\{\fnc{x}{\pp}\}_{(x,\pp)\in\Vman\times\ParMan}$ will define a unique continuous function on all of $\Vman\times\ParMan$. 

Notice that the projection $p$ takes the values in $(-\eps,\eps)$.
Therefore we get from\;\eqref{equ:local_formulas_for_shift_func} that
\begin{equation}\label{equ:Deltaxi_Deltaeta_2eps}
|\fncxp(\xi)-\fncxp(\eta)|<4\eps, \quad  \text{for all} \quad \xi,\eta\in\nbhxp. 
\end{equation}

\begin{claim}\label{clm:when_fncxp_coincides}
Suppose that $$|\fnc{x}{\pp}(\xi) - \fnc{x'}{\pp'}(\xi)|< 2\eps$$ for some $\xi\in\nbh{x}{\pp}\cap\nbh{x'}{\pp'}$.
Then $\fnc{x}{\pp} = \fnc{x'}{\pp'}$ on $\nbh{x}{\pp}\cap\nbh{x'}{\pp'}$.
\end{claim}
\begin{proof}
If $\eta\in\nbh{x}{\pp}\cap\nbh{x'}{\pp'}$, then
\begin{multline*}
|\fnc{x}{\pp}(\eta)-\fnc{x'}{\pp'}(\eta)| \ \leq \ 
|\fnc{x}{\pp}(\eta)-\fnc{x}{\pp}(\xi)| \ + \\ + \
|\fnc{x}{\pp}(\xi)-\fnc{x'}{\pp'}(\xi)| \ + \ 
|\fnc{x'}{\pp'}(\xi)-\fnc{x'}{\pp'}(\eta)| \ < \\ < \ 4\eps \ + \  \eps \ + \ 4\eps \ = \ 10\eps.
\end{multline*}
Hence by Claim\;\ref{clm:local_shift_func_is_unique} $\fnc{x}{\pp}(\eta)=\fnc{x'}{\pp'}(\eta)$.
\end{proof}

\begin{claim}\label{clm:ext_Delta_to_x_P}
$\fnc{x}{\pp} = \fnc{x}{\pp'}$ on $\nbh{x}{\pp}\cap\nbh{x}{\pp'}$.
\end{claim}
\begin{proof}
Since $\fnc{x}{\pp}$ and $\lmd{x}$ are continuous shift functions for $\Omega$ on the connected set $x\times \nbhxpP$, it follows from\;\eqref{equ:fncxp=lmdxp} and Corollary\;\ref{cor:local_shift_func_is_unique} that $\fnc{x}{\pp}=\lmd{x}$ on $x\times \nbhxpP$.
Hence
$$
\fnc{x}{\pp}(x,\pp') \ = \ 
\lmd{x}(\pp') \ \stackrel{\eqref{equ:fncxp=lmdxp}}{=\!=\!=} \ 
\fnc{x}{\pp'}(x,\pp') \ \text{for all} \ \pp'\in\nbhxpP.
$$
Then by Claim\;\ref{clm:when_fncxp_coincides} $\fnc{x}{\pp} = \fnc{x}{\pp'}$ on $\nbh{x}{\pp}\cap\nbh{x}{\pp'}$.
\end{proof}
Thus for each $x\in\Vman$ the functions $\{\fnc{x}{\pp}\}_{\pp\in\ParMan}$ give rise to a continuous shift function $\enc{x}:\mbh{x}\to\RRR$ for $\Omega$ on the open neighbourhood $\mbh{x}:=\cup_{\pp\in\ParMan} \nbhxp$ of $x\times\ParMan$ in $\Vman\times\ParMan$.
Since $\ParMan$ and each $\nbhxp$ is connected, we see that so is $\mbh{x}$.
Then it easily follows from Claim\;\ref{clm:local_shift_func_is_unique} and Corollary\;\ref{cor:local_shift_func_is_unique} that $\enc{x}$ is a unique shift function for $\Omega$ on $\mbh{x}$ such that $\enc{x}(\pp)=\lmd{x}(\pp)$ for at least one $\pp\in\ParMan$.

\begin{claim}\label{clm:ext_Delta_to_V_p0}
$\fnc{x}{\pp_0}=\fnc{x'}{\pp_0}$ on $\nbh{x}{\pp_0}\cap\nbh{x'}{\pp_0}$, whence $\enc{x}=\enc{x'}$ on $\mbh{x}\cap\mbh{x'}$.
\end{claim}
\begin{proof}
Notice that $\fnc{x}{\pp_0}$ and $\afunc$ are continuous shift functions for $\Omega$ on the connected set $\nbhX{x}{\pp_0}\times\pp_0$ and by assumption 
$$\fnc{x}{\pp_0}(x,\pp_0)=\lmd{x}(\pp_0)=\afunc(x).$$
Then by Corollary\;\ref{cor:local_shift_func_is_unique} 
$$\fnc{x}{\pp_0}(x',\pp_0)=\afunc(x')=\fnc{x}{\pp_0}(x',\pp_0) \ \text{for all} \ x'\in\nbhX{x}{\pp_0}.$$
Hence by Claim\;\ref{clm:when_fncxp_coincides} $\fnc{x}{\pp_0} = \fnc{x'}{\pp_0}$ on $\nbh{x}{\pp_0}\cap\nbh{x}{\pp_0}$.
\end{proof}

Thus the functions $\{\enc{x}\}_{x\in\Vman}$ define a continuous function on all of $\Vman\times\ParMan$ which coincides with $\Lambda$.
Theorem\;\ref{th:shift-function-for-deformations} is completed.
\end{proof}

\begin{corollary}\label{cor:local_existence_of_sh_func}
Let $\amap\in\EAV$, $z\in\Vman\setminus\FixF$, and $a\in\RRR$ be such that $\amap(z)=\AFlow(z,a)$.
Then there exists a neighbourhood $\Wman$ of $z$ and a unique continuous shift function $\afunc:\Wman\to\RRR$ for $\amap$ such that $\afunc(z)=a$.
In fact $\afunc$ is $\Cinf$.
\end{corollary}

\begin{corollary}\label{cor:shift-function-for-deformations}
Let $\Vman\subset\Mman$ be an open subset, $\VmanS=\Vman\setminus\FixF$, $\afunc\in\Ci{\VmanS\setminus\FixF}{\RRR}$, and $\Omega:\Vman\times\SmParMan\times\CntParMan\to\Mman$ be an $\dfrm{\SmParMan}{\CntParMan}{k}$-deformation in $\EAV$ such that $\Omega(x,\sx_0,\tau_0)=\AFlow(x,\afunc(x))$ for all $x\in\VmanS$.

If $\pi_1(\SmParMan\times\CntParMan)=0$, then there exists a unique $\dfrm{\SmParMan}{\CntParMan}{k}$-deformation
$\Lambda: \VmanS\times\SmParMan\times\CntParMan\to \RRR$
such that $\Lambda(x,\sx_0,\tau_0)=\afunc(x)$ for $x\in\VmanS$, and 
$\Omega(x,\sx,\tau)=\AFlow(x,\Lambda(x,\sx,\tau))$ for $(x,\sx,\tau)\in\VmanS\times\SmParMan\times\CntParMan.$
\end{corollary}

\section{Main result}\label{sect:main-results}
Let $\Vman\subset\Mman$ be an open connected set and $\Omega:\Vman\times I\to \Mman$ be a $k$-homotopy in $\EAV$ such that $\Omega_{0}=\ShAV(\afunc)$ for some $\Cinf$ function $\afunc:\Vman\to\RRR$.
Then by Corollary\;\ref{cor:shift-function-for-deformations} $\afunc$ extends to a continuous shift function $\Lambda:(\Vman\setminus\FixF)\times I\to \RRR$ for $\Omega$ such that $\Lambda_{0}=\afunc$.
But in general $\Lambda_{\cx}$ for $\cx>0$ can not be extended to a $\Cinf$ function on $\Vman$.

Such examples were constructed in\;\cite{Maks:MFAT:2009}.
There $\Omega$ is a $0$-homotopy being not a $1$-homotopy, $\Omega_{0}=\id$, $\Lambda_{0}=0$, and $\lim\limits_{z\to\orig}\Lambda_{1}(z)=+\infty$.
Our main result shows that this situation is typical: non-extendability of $\Lambda_{\cx}$ to a $\Cinf$ function almost always happens only for $0$-homotopies that are not $1$-homotopies and for special types of singularities.

Let $\CntParMan$ be a compact, simply connected and locally path connected topological space and $\org\in\CntParMan$.

\begin{theorem}\label{th:main_result}
Suppose $\FixF$ is nowhere dense in $\Vman$ and there exists $z\in\FixF$ such that either of the following conditions holds true:
\begin{enumerate}
 \item[\rm(a)]
$z$ is not a $\PN$-point;
 \item[\rm(b)]
$z$ is a $\PN$-point, $j^1\AFld(z)=0$, and $k\geq1$;
 \item[\rm(c)]
$z$ is a strong $\PN$-point and $k\geq1$.
\end{enumerate}

Let $\Omega:\Vman\times\CntParMan\to\Mman$ be a \myemph{$\dfrmc{\CntParMan}{k}$-deformation in $\imShAV$}, $\afunc\in\Ci{\Vman}{\RRR}$ be any shift function for $\Omega_{\org}$, and  $\Lambda:(\Vman\setminus\FixF)\times\CntParMan\to\RRR$ be a unique continuous shift function for $\Omega$ such that $\Lambda_{\org}=\afunc$, see Theorem\;\ref{th:shift-function-for-deformations}.
Then there exists a neighbourhood $\Wman$ of $z$ such that for each $\cx\in\CntParMan$ the function $\Lambda_{\cx}$ also extends to a $\Cinf$ function on $\Wman$, so $\Omega_{\cx}|_{\Wman}=\ShAW(\Lambda_{\cx}|_{\Wman})$.
\end{theorem}

Notice that the assumptions (a)-(c) do not include the case when $z$ is a $\PN$-point being not strong and such that $j^1\AFld(z)\not=0$.
Moreover, it is \myemph{not claimed} that $\Lambda$ becomes continuous on $\Vman\times\CntParMan$.

\begin{theorem}\label{th:suffcond_for_imShA_EidAk}
Let $\AFld$ be a vector field on $\Mman$ such that $\FixF$ is nowhere dense and let $k=0,\ldots,\infty$.
Suppose that for each $z\in\FixF$ there exists an open, connected neighbourhood $\Vman$ such that
$\imShAV=\EidAV{k}$
and either of the conditions {\rm(a)-(c)} of Theorem\;\ref{th:main_result} holds true.
Then \ $\imShA=\EidAFlow{k}.$
\end{theorem}
\begin{proof}
Let $\amap\in\EidAFlow{k}$, so there exists a $k$-homotopy $\Omega:\Vman\times I\to\Mman$ in $\EAV$ between the identity map $\Omega_0=\id_{\Mman}$ and $\Omega_1=\amap$.
By Theorem\;\ref{th:shift-function-for-deformations} there exists a unique $k$-homotopy $\Lambda:(\Vman\setminus\FixF)\times I \to \RRR$ being a shift function for $\Omega$ and such that $\Lambda_0\equiv 0$.
We will show that $\Lambda_{\cx}$ extends to a $\Cinf$ function on all of $\Mman$, whence $\Omega_{\cx}=\ShA(\Lambda_{\cx})$ for all $\cx\in I$.
This will imply $\imShA=\EidAFlow{k}.$

Let $z\in\FixF$ and $\Vman$ be a neighbourhood of $z$ such that $\imShAV=\EidAV{k}$.
Then $\Omega_{\cx}|_{\Vman}\in\EidAV{k}$, so $\Omega|_{\Vman\times I}$ is a $k$-homotopy in $\imShAV$.
Moreover, $\Lambda|_{(\Vman\setminus\FixF)\times I}$ is a unique extension of $\Lambda_{0}|_{\Vman}=0$.
By assumption either of the conditions (a)-(c) of Theorem\;\ref{th:main_result} holds true, whence for each $\cx\in I$ the function $\Lambda_{\cx}$ extends to a $\Cinf$ function on some neighbourhood of $z$.
Since $\FixF$ is nowhere dense such an extension is unique.
As $z$ is arbitrary, we obtain that $\Lambda_{\cx}$ smoothly extends to all of $\Mman$.
\end{proof}

For the proof of Theorem\;\ref{th:main_result} we need two preliminary results.

\subsection{Reduction of $1$-jets at singular points}\label{sect:reduct-1-jets}
Let $\Omega:\Vman\times\CntParMan\to\Mman$ be a $\dfrmc{\CntParMan}{k}$ deformation in $\imShAV$ and $\Lambda:(\Vman\setminus\FixF)\times\CntParMan\to\RRR$ be a shift function for $\Omega$.
Let also $z\in\FixF\cap\Vman$.
Choose local coordinates $(x_1,\ldots,x_m)$ at $z$ in which $z=\orig\in\RRR^m$.
Let $A$ be an $(m\times m)$-matrix being the linear part of $\AFld$ at $z$.
Then we have the exponential map 
$$
e:\RRR\to \mathrm{GL}(m,\RRR), \qquad e(t) = \exp(At).
$$
Denote its image by $\mathcal{R}_{A}$.
If $A\not=0$, then $e$ is an immersion, i.e. a local homeomorphism onto its image.

For each $\cx\in\CntParMan$ denote by $J_{\cx}$ the Jacobi matrix of $\Omega_{\cx}$ at $z$.
By assumption $\Omega_{\cx}=\ShAV(\afunc_{\cx})$ for some $\afunc_{\cx}\in\Ci{\Vman}{\RRR}$.
Then it is easy to show, \cite[Th.\;5.1]{Maks:CEJM:2009}, that
$$J_{\cx} = \exp(A\,\afunc_{\cx}(z)),$$
so we have a well-defined map 
$$
\eta:\CntParMan \to \mathcal{R}_{A}, \qquad 
\eta(\cx) = J_{\cx}.
$$
If $k\geq1$, then $\eta$ is continuous.
\begin{lemma}\label{lm:deform_in_imShAV}
Suppose that $\eta$ is continuous and lifts to a continuous map $\widetilde{\eta}:\CntParMan\to\RRR$ making the following diagram commutative:
$$
\xymatrix{
                         & \RRR \ar[d]^{e} \\
  \CntParMan \ar@{.>}[ru]^{\widetilde{\eta} } \ar[r]_{\eta} & \mathcal{R}_{A}
 }
$$

The latter holds for instance when $\CntParMan$ is compact, path connected, and simply connected, e.g. $\CntParMan=I^{\dd}$.

Let $\Wman\subset\Vman$ be an arbitrary small neighbourhood of $z$.
Then there exists a function $\nu:\Vman\times\CntParMan\to\RRR$ being a $\dfrmc{\CntParMan}{\infty}$ deformation in $\Ci{\Vman}{\RRR}$ and vanishing outside $\Wman\times\CntParMan$ such that the $\dfrmc{\CntParMan}{k}$-deformation
$$
\Omega':\Vman\times\CntParMan\to\Mman,
\qquad
\Omega'(x,\cx)=\AFlow(\Omega(x,\cx),-\nu(x,\cx))
$$
has the following properties: for each $\cx\in\CntParMan$
\begin{enumerate}
 \item
$j^1\Omega'_{\cx}(z)$ is the identity, and
 \item
the function $\Lambda'_{\cx}=\Lambda_{\cx} - \nu_{\cx}$ is a shift function for $\Omega'$.
\end{enumerate}
In particular, $\Lambda_{\cx}$ extends to a $\Cinf$ function near $z$ iff so does $\Lambda'_{\cx}$.
\end{lemma}
\begin{proof}
Let $\mu:\Vman\to[0,1]$ be a $\Cinf$ function such that $\mu=1$ on some neighbourhood of $z$ and $\mu=0$ on $\overline{\Vman\setminus\Wman}$.
Define $\nu(x,\cx) = \mu(x)\widehat{\eta}(\cx)$.
Then it is easy to verify that $\nu$ satisfies the statement of our lemma.
\end{proof}

\subsection{Smoothness of shift functions}\label{sect:smmothness_sh_func}
In this section we will assume that $\Vman\subset\Mman$ is an open connected subset such that $\VmanS=\Vman\setminus\FixF$ is connected and the shift map $\ShAVS$ is periodic.
Let $\theta:\VmanS\to(0,+\infty)$ be the positive generator of $\ZidAVS$.

Define the map $\dif:\Vman\to\Mman$ by
\begin{equation}\label{equ:h_F_x_mup}
\dif(x)=\begin{cases} 
 \AFlow(x,\theta(x)/2), & x\in\VmanS=\Vman\setminus\FixF, \\
 x, & x\in\FixF\cap\Vman.
\end{cases}
\end{equation}
Evidently, $\dif$ is continuous on $\VmanS$ but in general is discontinuous on $\FixF\cap\Vman$.
If $x\in\VmanS$, then $\theta(x)=n_{x}\Per(x)$ for some $n_{x}\in\ZZZ$.
Hence $\dif(x)=x$ iff $n_x$ is even, see Figure\;\ref{fig:x_hx}. 
\begin{figure}[ht]
\includegraphics[height=1.5cm]{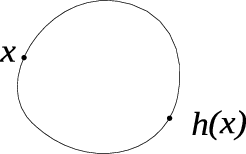}
\caption{$\theta(x)=n_{x}\Per(x)$, $n_x$ is odd}
\protect\label{fig:x_hx}
\end{figure}

\newcommand\COND[1]{$(\mathbf{\mathsf{#1}})$}

\newcommand\CONDhzdiscont{\COND{HD}}
\newcommand\CONDFonejetvanish{\COND{JF0}}
\newcommand\CONDthetainfty{\COND{TI}}

\newcommand\CONDOmegazeroid{\COND{\Omega\Lambda}}
\newcommand\CONDOmegaonejetid{\COND{J\Omega id}}
\newcommand\CONDhcont{\COND{HC}}

\newcommand\CONDPerxiinfty{\COND{PU}}
\newcommand\CONDthetaxiinfty{\COND{TU}}
\newcommand\CONDOmegatinothi{\COND{\Omega H}}
\newcommand\CONDdiamorb{\COND{D}}

\newcommand\CONDLambdatsmooth{\COND{Z}}

\newcommand\metr[2]{\|#1,#2\|}
\newcommand\CL{L}

\begin{theorem}\label{th:smooth_ext_shift_func}
Let $\Omega:\Vman\times \CntParMan\to\Mman$ be a $\dfrmc{\CntParMan}{1}$-deformation in $\imShAV$, $\Lambda_{\org}\in\Ci{\Vman}{\RRR}$ be any shift function for $\Omega_{\org}$, and let also $\Lambda:\VmanS\times \CntParMan\to\RRR$ be a unique continuous extension of $\Lambda_{\org}$ being a shift function for $\Omega$.
Suppose that there exists a point $z\in\FixF\cap\Vman$ satisfying one of the following conditions:
\begin{enumerate}
\item[\CONDhzdiscont]
$\dif$ is discontinuous at $z$;
\item[\CONDFonejetvanish]
 $j^1\AFld(z)=0$;
\item[\CONDthetainfty]
 $\lim_{\VmanS\ni x\to z}\limits\theta(x)=\infty$.
\end{enumerate}
Then
\begin{enumerate}
 \item[\CONDLambdatsmooth] 
for each $\cx\in \CntParMan$ the function $\Lambda_{\cx}:\VmanS\to\RRR$ extends to a $\Cinf$ function on all of $\Vman$ as well, however the induced extension $\Lambda:\Vman\times \CntParMan\to\RRR$ of $\Lambda$ is not claimed to be even continuous.
\end{enumerate}
\end{theorem}
\begin{proof}
First notice that if $\nu:\Vman\times \CntParMan\to\RRR$ is any $\dfrmc{\CntParMan}{k}$-deformation in $\Ci{\Vman}{\RRR}$, then the following map
\begin{equation}\label{equ:Omegapr_F_Om__nu}
\Omega':\Vman\times \CntParMan\to \Mman,
\qquad 
\Omega'(x,t) = \AFlow(\Omega(x,\cx),-\nu(x,\cx))
\end{equation}
is a $\dfrmc{\CntParMan}{k}$-deformation as well as $\Omega$.
Moreover, the function 
$$\Lambda'=\Lambda-\nu:(\Vman\setminus\FixF)\times \CntParMan\to\RRR$$
is a shift function for $\Omega'$ on $\Vman\setminus\FixF$ and $\Lambda'_{\org}=\Lambda_{\org}-\nu_{\org}$ extends to a $\Cinf$ function $\afunc_{\org}-\nu_{\org}$ on all of $\Vman$.
Hence $\Lambda'_{\cx}=\Lambda_{\cx}-\nu_{\cx}$ extends to a $\Cinf$ function on $\Vman$ iff so does $\Lambda_{\cx}$.
Thus one may replace $\Omega$ with $\Omega'$.

In particular, we can reduce the problem to the case when $\Omega$ satisfies the following additional assumptions:
\begin{enumerate}
 \item[\CONDOmegazeroid] 
$\Omega_{\org}=i_{\Vman}:\Vman\subset\Mman$ is the identity inclusion and $\Lambda_{\org}\equiv0$.
 \item[\CONDOmegaonejetid]
$j^1\Omega_{\cx}(z)=\id$ for all $\cx\in \CntParMan$.
\end{enumerate}
To establish \CONDOmegazeroid\ put $\nu(x,\cx)=\afunc_{\org}(x)$.
Then $\Lambda'_{\org}=0$, and $\Omega'_{\org}=i_{\Vman}$.

Furthermore, suppose that \CONDOmegazeroid\ holds for $\Omega$.
Since $\Omega(z,\cx)=z$ for all $\cx\in \CntParMan$, we obtain from Lemma\;\ref{lm:deform_in_imShAV} that there exists even a $\dfrmc{\CntParMan}{\infty}$-deformation $\nu:\Vman\times \CntParMan\to\RRR$ such that $\nu_{\org}=0$ and the $\dfrmc{\CntParMan}{1}$-deformation \eqref{equ:Omegapr_F_Om__nu} satisfies \CONDOmegaonejetid.

\medskip

We will now introduce more conditions. 

\begin{enumerate}
 \item[\CONDhcont]
$\dif$ is continuous on some neighbourhood of $z$. 
\end{enumerate}

Let $\{x_i\}_{i\in\NNN}\subset\VmanS$ be a sequence converging to $z$.
Then we also consider the following conditions:

\begin{enumerate}
 \item[\CONDPerxiinfty]
the sequence of periods $\{\Per(x_i)\}_{i\in\NNN}$ is unbounded;
 \item[\CONDthetaxiinfty]
the sequence $\{\theta(x_i)\}_{i\in\NNN}$ is unbounded;
 \item[\CONDOmegatinothi]
$\Omega_{\cx}(x_i)\not=\dif(x_i)$ for all $\cx\in \CntParMan$ and all sufficiently large $i\in\NNN$;
 \item[\CONDdiamorb]
there exists $\CL>0$ and a Riemannian metric $\metr{\cdot}{\cdot}$ on some neighbourhood of $z$ such that 
$$\metr{x_i}{z} <  \CL \cdot  \metr{x_i}{\dif(x_i)},\qquad i\in\NNN.$$
\end{enumerate}

\begin{lemma}\label{lm:implications}
The following implications hold true:

1) \CONDOmegazeroid\ \ $\&$  \ \CONDthetaxiinfty\ \ $\&$ \ \CONDOmegatinothi\ \ $\Rightarrow$ \ \CONDLambdatsmooth;


2) \CONDhzdiscont\ \ $\Rightarrow$ \ $\exists\{x_i\}_{i\in\NNN}$ satisfying $\lim\limits_{i\to\infty} x_i=z$,  \ \CONDPerxiinfty, \ and \ \CONDOmegatinothi;


3) \CONDhcont\ $\Rightarrow$  \ $\exists\{x_i\}_{i\in\NNN}$ satisfying $\lim\limits_{i\to\infty} x_i=z$ \ and \ \CONDdiamorb;


4) \CONDFonejetvanish\ \ $\&$ \ \CONDdiamorb\ \ $\Rightarrow$ \  \CONDPerxiinfty;


5) \CONDhcont\ \ $\&$ \ \CONDOmegaonejetid\ \ $\&$ \ \CONDdiamorb \ $\Rightarrow$ \  \CONDOmegatinothi;


6) 
\CONDthetainfty\ \ $\Rightarrow$ \ \CONDthetaxiinfty.


7) \CONDPerxiinfty\ \ $\Rightarrow$ \ \CONDthetaxiinfty.
\end{lemma}

Assuming that this lemma is proved let us complete Theorem\;\ref{th:smooth_ext_shift_func}.
We have to show that either of the conditions \CONDhzdiscont, \CONDFonejetvanish, or \CONDthetainfty\ together with \CONDOmegazeroid\ and \CONDOmegaonejetid\ implies \CONDLambdatsmooth.

The implication \CONDhzdiscont\ \& \CONDOmegazeroid\ \ $\Rightarrow$ \ \CONDLambdatsmooth\ follows from 2), 7) and 1) of Lemma\;\ref{lm:implications}.

Further, if \CONDhzdiscont\ fail for any $z\in\Vman\cap\FixF$, that is $\dif$ is continuous at each such $z$, then $\dif$ is continuous on all of $\Vman$, so \CONDhcont\ holds true.

Then \CONDhcont\ $\&$  \CONDFonejetvanish\ $\&$ \CONDOmegazeroid\ $\&$ \CONDOmegaonejetid\ $\Rightarrow$ \CONDLambdatsmooth\ 
by 4), 7), 5), and 1) of Lemma\;\ref{lm:implications}.

Moreover, \CONDhcont\ $\&$  \CONDthetainfty\ $\&$ \CONDOmegazeroid\ $\&$ \CONDOmegaonejetid\ $\Rightarrow$ \CONDLambdatsmooth\ by 3), 6), and 1) of Lemma\;\ref{lm:implications}. 

This completes Theorem\;\ref{th:smooth_ext_shift_func} modulo Lemma\;\ref{lm:implications}.
\end{proof}

\subsection{Proof of Lemma\;\ref{lm:implications}}
The implications 6) and 7) are trivial.

1) \CONDOmegazeroid\ \ $\&$  \ \CONDthetaxiinfty\ \ $\&$ \ \CONDOmegatinothi\ \ $\Rightarrow$ \ \CONDLambdatsmooth.

Thus $\Omega_{\org}=i_{\Vman}:\Vman\subset\Mman$, $\Lambda_{\org}\equiv0$, and there exists a sequence $\{x_i\}_{i\in\NNN}$ which converges to some $z\in\FixF\cap\Vman$ and satisfies
$$
\lim\limits_{i\to\infty}\theta(x_i)=+\infty,
\leqno{{\text{\CONDthetaxiinfty}}} 
$$
$$
\Omega_{\cx}(x_i) \not = \dif(x_i), \qquad i\in\NNN, \ \cx\in \CntParMan.
\leqno{{\text{\CONDOmegatinothi}}}  
$$
We have to show that for each $\cx\in \CntParMan$ the function $\Lambda_{\cx}$ extends to a $\Cinf$ function on all of $\Vman$.

Notice that\;\CONDthetaxiinfty\ implies that $\ShAV$ is non-periodic.
Indeed, suppose $\ShAV$ is periodic and let $\nu$ be the positive generator of $\ZidAV$.
Then by Lemma\;\ref{lm:per-non-per} $\nu=\theta$ on $\VmanS$, whence $\theta$ is bounded on $\VmanS$ which contradicts to \CONDthetaxiinfty.

Thus $\ShAV$ is non-periodic, whence for each $\Omega_{\cx}\in\imShAV$, $(\cx\in \CntParMan)$, there exists a unique $\Cinf$ shift function $\afunc_{\cx}\in\Ci{\Vman}{\RRR}$.
In particular, $\Lambda_{\org}=\afunc_{\org}$.
We will show that $\Lambda_{\cx}=\afunc_{\cx}$ on $\VmanS$ for all $\cx\in \CntParMan$.

Since $\afunc_{\cx}|_{\VmanS}$ and $\Lambda_{\cx}$ are shift functions for $\Omega_{\cx}$ on $\VmanS$, we obtain from Lemma\;\ref{lm:ker1} that
$$
\afunc_{\cx}-\Lambda_{\cx}=k_{\cx}\, \theta
$$
on $\VmanS$ for some $k_{\cx}\in\ZZZ$.
It suffices to show that $k_{\cx}=0$.

Condition \CONDOmegatinothi\ means that the image of the map $\omega_i:\CntParMan \to \orb_{x_i}$ defined by $$\omega_i(\cx)=\AFlow(x_i,\Lambda_{\cx}(x_i))=\Omega(x_i,{\cx})$$ does not contain the point $\dif(x_i)$.
Moreover, since $\Omega_{\org}(x_i)=x_i \not=\dif(x_i)$ and $\Lambda_{\org}=0$, it follows from the construction of $\Lambda$ that 
\begin{equation}\label{equ:Lambda_1_2_theta}
|\Lambda_{\cx}(x_i) - \Lambda_{\org}(x_i) | = |\Lambda_{\cx}(x_i) |< \tfrac{1}{2}\, \Per(x_i) \leq \tfrac{1}{2}\theta(x_i)
\end{equation}
for all $\cx\in \CntParMan$ and $i\in\NNN$.

On the other hand we get from \CONDthetaxiinfty\ and continuity of $\afunc_{\cx}$ that
$$
|\afunc_{\cx}(x_i) - \afunc_{\org}(x_i)| = |\afunc_{\cx}(x_i)|  < \tfrac{1}{2}\theta(x_i)
$$
for all sufficiently large $i\in\NNN$.

Hence
\begin{multline*}
k_{\cx}\theta(x_i) = |\Lambda_{\cx}(x_i) - \afunc_{\cx}(x_i) | \leq	 
|\Lambda_{\cx}(x_i)| + |\afunc_{\cx}(x_i)| < \\ < \tfrac{1}{2}\Per(x_i) + \tfrac{1}{2}\theta(x_i) \leq \theta(x_i).
\end{multline*}
This implies $k_{\cx}=0$ and thus $\Lambda_{\cx}=\afunc_{\cx}$ on $\VmanS$.

\medskip 

2) \CONDhzdiscont\ \ $\Rightarrow$ \ $\exists\{x_i\}_{i\in\NNN}$ satisfying $\lim\limits_{i\to\infty} x_i=z$, \ \CONDPerxiinfty, \ and \ \CONDOmegatinothi.

Discontinuity of $\dif$ (that is property \CONDhzdiscont) at $z$ means that there exists a sequence $\{x_i\}_{i\in\NNN}$ converging to $z$ such that $\{\dif(x_i)\}_{i\in\NNN}$ lays outside some neighbourhood $\Wman$ of $z$.

\CONDOmegatinothi.
Since $\Omega(z\times\CntParMan)=z$ and $\CntParMan$ is compact, we can assume that $\{\Omega_{\cx}(x_i)\}_{i\in\NNN}\subset\Wman$ for all $i\in\NNN$ and $\cx\in \CntParMan$, therefore $\Omega_{\cx}(x_i)\not=\dif(x_i)$.

\CONDPerxiinfty.
Suppose \CONDPerxiinfty\ fails.
Then there exists $C>0$ such that $\Per(x_i)<C$ for all $i$.
Since $z$ is a fixed point of $\AFlow$, there exists another neighbourhood $\Uman$ of $z$ such that $\AFlow(\Uman\times[0,C])\subset\Wman$.
Therefore (passing to a subsequence) we can assume that $x_i\in\Uman$ for all $i\in\NNN$, hence
$$\dif(x_i) \ = \ \AFlow(x_i,\Per(x_i)/2) \ \in \ \AFlow(\Uman\times[0,C]) \ \subset \ \Wman$$
which contradicts to the assumption $\dif(x_i)\not\in\Wman$.

\medskip 

3) \CONDhcont\ $\Rightarrow$  \ $\exists\{x_i\}_{i\in\NNN}$ satisfying $\lim\limits_{i\to\infty} x_i=z$ \ and \ \CONDdiamorb.

Thus we have that $\dif$ is continuous on some neighbourhood $\Wman\subset\Vman$ of $z$.
Our aim is to show that there exist 
\begin{itemize}
\item
a neighbourhood $\Uman\subset\Wman$ of $z$ such that $\dif(\Uman)=\Uman$, 
\item
a Riemannian metric $d$ on $\Uman$,
\item 
a sequence $\{x_i\}_{i\in\NNN} \subset\Uman\setminus\FixF$ converging to $z$, and 
\item 
a number $\CL>0$
\end{itemize}
satisfying
$$ 
 \metr{z}{x_i} < \CL \cdot  \metr{x_i}{\dif(x_i)}.
\leqno{\text{\CONDdiamorb}}
$$

First suppose that $z\in\Int\Mman$.
Then we can assume that $\Wman$ is an open subset of $\RRR^n$ and $d$ is the induced Euclidean metric on $\Wman$.
Since $\dif$ is continuous on $\Wman$ and $\dif(z)=z$, we can find an open ball $B_s\subset\Wman$ of some radius $s>0$ centered at $z$ such that $\dif(B_s)\subset\Wman$.
Then
$$ \Uman_s := B_s \cup \dif(B_s) \ \subset \ \Wman.$$
By assumption $\VmanS$ is connected, therefore by (3) of Theorem\;\ref{th:ker2} the restriction of $\theta$ to $\orb_x\cap\VmanS$ is constant for any $x\in\VmanS$.
This easily implies that 
$$\dif^{2}(x)=\AFlow(x, \theta(x)) =x, \qquad x\in B_s,$$
so $\dif$ yields a $\ZZZ_2$-action on $\Uman_s$.
Due to (2) of Theorem\;\ref{th:ker2} this action is non-trivial and therefore effective.

By a well known theorem of M.\;Newman\;\cite{Newman:QJM:1931} all the orbits of a $\ZZZ_2$-action on a manifold $\Uman_s$ can not be arbitrary small.
D.\;Hoffman and L.\;N.\;Mann\;\cite[Th.\;1]{HoffmanMann:PAMS:76}, using a result of A.\;Dress\;\cite{Dress:Topol:1969}, obtained lower bounds for diameters of such actions.
Denote by $\bar r$ the radius of convexity of $\Uman_s$ at $z$.
Since $B_s\subset\Uman_s$, it follows that $s\leq \bar r$. 
Then it was shown in\;\cite[Th.\;1]{HoffmanMann:PAMS:76} that there exists an orbit or this action of diameter greater than $s/2$, that is 
$$
 \metr{x_s}{\dif(x_s)}> s/2
$$
for some $x_s\in\Uman_s$.
Interchanging $x_s$ and $\dif(x_s)$, if necessary, we can assume that in fact $x_s\in B_s$.
Hence 
\begin{equation}\label{equ:d_z_xs__T_d_xs_hasxs}
\metr{z}{x_s} < s < \CL \cdot \metr{x_s}{\dif(x_s)}, \qquad \text{where $L=2$.}
\end{equation}

Decreasing $s$ to $0$ we will find a sequence $\{x_{s_i}\}_{i\in\NNN}$ satisfying \CONDdiamorb.

A more deep analysis of the proofs of\;\cite[Th.\;1]{HoffmanMann:PAMS:76} and\;\cite[Lm.\;3]{Dress:Topol:1969} shows that similar estimations but with another constant $\CL$ hold in the case when $z\in\partial\Mman$, see\;\cite{Maks:period_functions} for details.

\medskip 

4) \CONDFonejetvanish\ \ $\&$ \ \CONDdiamorb\ \ $\Rightarrow$ \  \CONDPerxiinfty.

Thus $j^1\AFld(z)=0$ and we have a sequence $\{x_i\}_{i\in\NNN}$ converging to $z$ and satisfying \CONDdiamorb.
We will show that there is subsequence $\{x_{i_k}\}_{k\in\NNN}$ such that $\lim\limits_{k\to\infty}\Per(x_{i_k})=+\infty$.

Suppose that there exists $C>0$ such that $\Per(x_i)<C$ for all $i$.
Since $z$ is a fixed point of $\AFlow$, we can decrease $\Uman$ and assume that $\AFlow(\Uman\times[0,C])\subset\Vman$.
Moreover, as $j^1\AFld(z)=0$, we can also suppose that there exists $A>0$ such that 
$$|\AFld(x)|\leq A \cdot \metr{z}{x}^2, \quad \text{for all $x\in\Uman$}.$$

Notice that the length $l(x_i)$ of the orbit of $x_i$ can be calculated by the following formula:
$$ l(x_i) = \int_{0}^{\Per(x_i)}\,|\AFld(\AFlow(x_i,t))|dt,$$
whence 
$$l(x_i) \leq \Per(x) \cdot \sup_{t\in[0,\Per(x)]} \, |\AFld(\AFlow(x,t))| \leq C \cdot A\cdot \metr{z}{x_i}^2.$$
Therefore
\begin{multline*}
\metr{z}{x_i}  < \CL \cdot \metr{z}{\dif(x_i)} \leq \CL \cdot l(x_i) \leq \CL \cdot C \cdot A\cdot \metr{z}{x_i}^2,
\end{multline*}
whence $0< \tfrac{1}{\CL\cdot C \cdot A} \leq \metr{z}{x_i}$, which contradicts to the assumption that $\{x_i\}_{i\in\NNN}$ converges to $z$.

\medskip 

5) \CONDOmegaonejetid\ \ $\&$ \ \CONDdiamorb \ $\Rightarrow$ \  \CONDOmegatinothi;

Thus $j^1\Omega_{\cx}(z)=\id$ for all $\cx\in \CntParMan$ and there is a sequence $\{x_i\}_{i\in\NNN}$ converging to $z$ and satisfying $\metr{z}{x_i}< \CL \cdot \metr{x_i}{\dif(x_i)}$ for some $\CL>0$.
We have to show that $\Omega_{\cx}(x_i)\not=\dif(x_i)$ for all $\cx\in \CntParMan$ and all sufficiently large $i$.

Since $\Omega$ is a $\dfrmc{\CntParMan}{1}$-deformation, it follows from \CONDOmegaonejetid\ that there exists a continuous function $\gamma:\Uman\to\RRR$ such that $\gamma(z)=0$ and 
$$
\metr{x}{\Omega_{\cx}(x)} \leq \gamma(x) \cdot \metr{z}{x}
$$
for all $(x,\cx)\in \Uman\times \CntParMan$, see\;\cite{Maks:period_functions}.
If $\Omega$ were a $\Cont{2}$ map or at least a $\dfrmc{\CntParMan}{2}$-deformation, we would have a usual estimation:
$$
\metr{x}{\Omega_{\cx}(x)} \leq C \cdot \metr{z}{x}^2
$$
with some constant $C>0$.

Then 
\begin{multline*}
\metr{\dif(x_i)}{\Omega_{\cx}(x_i)} \geq 
\metr{x_i}{\dif(x_i)} -
\metr{x_i}{\Omega_{\cx}(x_i)} > \\ >
\tfrac{1}{\CL}\, \metr{z}{x_i} - \gamma(x_i,t) \cdot \metr{z}{x_i}  =
\bigl(\tfrac{1}{\CL}-\gamma(x_i)\bigr) \cdot \metr{z}{x_i}.
\end{multline*}
Notice that $\lim\limits_{i\to\infty}\gamma(x_i)=\gamma(z)=0$.
Therefore (passing if necessary to a subsequence) we can assume that $\tfrac{1}{\CL}-\gamma(x_i,\cx)>0$ for all $i$, whence $\metr{\dif(x_i)}{\Omega_{\cx}(x_i)}>0$.
This means that $\Omega_{\cx}(x_i)\not=\dif(x_i)$.

Lemma\;\ref{lm:implications} is completed.

\subsection{Proof of Theorem\;\ref{th:main_result}}
\label{sect:proof:th:main_result}
Let $\Vman$ be a neighbourhood of $z$ such that $\FixF$ is nowhere dense in $\Vman$, $\Omega:\Vman\times \CntParMan\to\Mman$ be a $\dfrmc{\CntParMan}{k}$-deformation in $\imShAV$, and $\Lambda:\VmanS\times \CntParMan\to\Mman$ be a continuous shift function for $\Omega$ such that $\Lambda_{\org}$ extends to a $\Cinf$ function $\afunc$ on all of $\Vman$.
We have to prove that there exists a neighbourhood $\Wman$ of $z$ such that for each $\cx\in \CntParMan$ the function $\Lambda_{\cx}$ extends to a $\Cinf$ function on $\Wman$.

By assumption $\Omega_{\cx}\in\imShAV$, so $\Omega_{\cx}=\ShAV(\afunc_{\cx})$ for some $\afunc_{\cx}\in\Ci{\Vman}{\RRR}$.
Hence $\Lambda_{\cx} = \afunc_{\cx} + \mu_{\cx}$ on $\VmanS$ for some $\mu_{\cx}\in\ZidAVS$.

\smallskip 

(a) Suppose $z$ is not a $\PN$-point.
Then we can find a neighbourhood $\Wman$ of $z$ such that either 
\begin{itemize}
 \item[(a$'$)]
$\ShAWS$ is non-periodic, or
 \item[(a$''$)]
both $\ShAW$ and $\ShAWS$ are periodic,
\end{itemize}
where $\WmanS=\Wman\setminus\FixF$.

In the case (a$'$) we have $\ZidAWS=\{0\}$.
In particular, $\mu_{\cx}=0$, and therefore $\Lambda_{\cx} = \afunc_{\cx}$ on $\WmanS$. 
Hence $\Lambda_{\cx}$ extends to a $\Cinf$ function $\afunc_{\cx}$ on $\Wman$.

In the case (a$''$) due to Lemma\;\ref{lm:per-non-per} we have that $\ZidAW=\{n\theta\}_{n\in\ZZZ}$ for some $\theta\in\Ci{\Wman}{\RRR}$ and $\ZidAWS=\{n\theta|_{\WmanS}\}_{n\in\ZZZ}$.
Hence $\mu_{\cx}=k_{\cx}\theta|_{\WmanS}$ for some $k_{\cx}\in\ZZZ$.
Therefore again $\Lambda_{\cx}$ extends to a $\Cinf$ function $\afunc_{\cx}+k_{\cx} \theta$ on all of $\Wman$.

Finally, in the cases (b) and (c) $z$ is $\PN$-point, so there exists a neighbourhood $\Wman$ of $z$ such that $\ShAW$ is non-periodic, while $\ShAWS$ is periodic.
Then in the case (b) (resp. (c)) the restriction $\Omega|_{\Wman\times \CntParMan}$ satisfies assumptions of Theorem\;\ref{th:smooth_ext_shift_func} and condition \CONDFonejetvanish\ (resp. \CONDthetainfty).
Hence for each $\cx\in \CntParMan$ the function $\Lambda_{\cx}$ also smoothly extends to all of $\Wman$.
\qed

\section{Decreasing $\Vman$}\label{sect:desreasing_V}
The following lemma shows that the relation $\imShAV=\EidAV{k}$ is preserved if $\Vman$ is properly decreased.
\begin{lemma}\label{lm:reduce_imShV_EidAVk}
Let $\Vman \subset \Mman$ be an open connected subset such that $\Vman\cap\FixF$ is nowhere dense in $\Vman$ and $\imShAV=\EidAV{k}$ for some $k\geq0$.
Let also $\Wman\subset\Vman$ be an open connected subset such that $\Wman\cap\FixF$ is \myemph{closed\/} (and so \myemph{open-closed}) in $\Vman\cap\FixF$.
Then $\imShAW=\EidAW{k}$.
\end{lemma}
\begin{proof}
Let $\Omega:\Wman\times I \to \Mman$ be a $k$-homotopy such that $\Omega_0=i_{\Wman}:\Wman\subset\Mman$ is the identity inclusion.
We have to show that for every $t\in I$ there exists a $\Cinf$ function $\afunc_{t}:\Wman\to\RRR$ such that $\dif=\ShAW(\afunc)$.

Let $\Lambda:(\Wman\setminus\FixF)\times I\to\RRR$ be a unique $k$-homotopy being a shift function for $\Omega$ and such that $\Lambda_{0}\equiv0$.
We will show that every $\Lambda_{t}$ either extends to a $\Cinf$ function on all of $\Wman$, or can be replaced with some other shift function of $\Omega_{0}$ which smoothly extends to $\Wman$.

First we extend $\Lambda$ to $(\Vman\setminus\FixF)\times I$ changing it outside $(\Wman\cap\FixF)\times I$.

Since $\Wman\cap\FixF$ is open-closed in $\Vman\cap\FixF$, there exists a $\Cinf$ function $\mu:\Vman\to[0,1]$ such that $\mu=1$ on some open neighbourhood $\Nman\subset\Vman$ of $\Wman\cap\FixF$ and $\mu=0$ on some open neighbourhood of $\Vman\setminus\Wman$, see Figure\;\ref{fig:vw}.
\begin{figure}[ht]
\includegraphics[height=2cm]{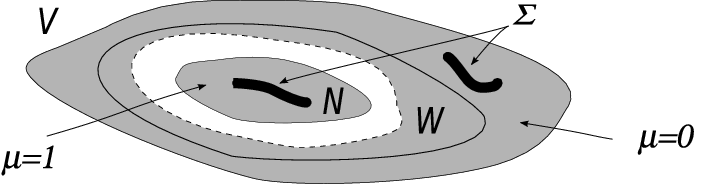}
\protect\caption{}\label{fig:vw}
\end{figure}
Then it follows from Lemma\;\ref{lm:change_deform} that a function $\Lambda':(\Vman\setminus\FixF)\times I\to\RRR$ defined by
$$
\Lambda'(x,t)=
\left\{
\begin{array}{cl}
\mu(x)\Lambda(x,t), & x\in\Wman\setminus\FixF, \\
0,                  & x\in\Vman\setminus\Wman
\end{array}
\right.
$$
is a $k$-homotopy which coincides with $\Lambda$ on $(\Nman\setminus\FixF)\times I$.
Then the following map $\Omega':\Vman\times I\to\Mman$ defined by
$$
\Omega'(x,t)=
\left\{
\begin{array}{cl}
\AFlow(x,\Lambda'(x,t)), & x\in\Vman\setminus(\Wman\cap\FixF), \\
x, & x\in\Wman\cap\FixF
\end{array}
\right.
$$
is a $k$-homotopy in $\EidAV{k}$ which coincides with $\Omega$ on $\Nman\times I$.

Since $\imShAV=\EidAV{k}$, it follows that $\Omega'_{t}$ has a $\Cinf$ shift function $\afunc'_{t}$ on $\Vman$.
Thus 
$$
\Omega'_{t}(x) = \AFlow(x,\mu(x)\Lambda_{t}(x)) = \AFlow(x,\afunc'_{t}(x)),
\qquad x\in\Vman\setminus\FixF.
$$
Denote $\VmanS=\Vman\setminus\FixF$.
Then $\mu\Lambda_{t}$ is a shift function for $\Omega'_{t}$ on $\VmanS$, while $\afunc'_{t}$ is a shift function for $\Omega'_{t}$ on $\Vman$.
Consider three cases.

1) If $\ShAV$ and $\ShAVS$ are non-periodic, then it follows from Lemma\;\ref{lm:ker1} that $\afunc'_{t}=\Lambda'_{t}=\mu\Lambda_{t}$ on $\VmanS$.
Since $\mu=1$ near $\Wman\cap\FixF$ and $\Lambda_{t}$ is smooth outside this set, we put $\Lambda_{t}=\afunc_t$ on $\Wman\cap\FixF$.
Then $\Lambda_t$ becomes smooth on all of $\Wman$.

2) If both $\ShAV$ and $\ShAVS$ are periodic, then by Lemmas\;\ref{lm:ker1} and\;\ref{lm:per-non-per} $\mu\Lambda_{t}=\afunc'_{t}+ l_{t}\theta$ on $\VmanS$, where $\theta:\Vman\to(0,\infty)$ is the positive generator of the kernel of $\ShAV$ and $l_{t}\in\ZZZ$. 
Then similarly to the case 1) $\Lambda_{t}$ smoothly extends to all of $\Wman$.

3) Suppose that $\ShAVS$ is periodic while $\ShAV$ is non-periodic, and let $\theta:\VmanS\to(0,\infty)$ be the positive generator of the kernel of $\ShAVS$.
Then $\afunc'_{t}=\mu\Lambda_{t}-l_{t}\theta$ on $\VmanS$ for some $l_{t}\in\ZZZ$.
Again $\afunc_t=\Lambda_{t}-l_{t}\theta$ smoothly extends to all of $\Wman$.

We claim that $\afunc_t$ is a shift function for $\Omega_{t}$.
Indeed, we have that $\Omega_{t}(x)=\AFlow_{t}(x,\Lambda_{t}(x))$ and $\AFlow(x,\theta(x))=\AFlow(x,-l_{t}\,\theta(x))=x$, whence
$$
\AFlow(x,\Lambda_{t}(x)-l_{t}\theta(x))=
\AFlow\bigl(\AFlow(x,-l_{t}\theta(x)), \Lambda_{t}(x) \bigr)=
\AFlow(x, \Lambda_{t}(x))=\Omega_{t}(x).
$$ 
Thus $\Omega_{t}=\ShAW(\Lambda_{t}-l_{t}\theta) \in \imShAW$.
\end{proof}

\begin{corollary}\label{cor:reduce_imShV_EidAVk}
Suppose $z$ is an isolated singular point of $\AFld$, and let $\Vman$ be a connected, open neighbourhood of $z$ such that $\Vman\cap\FixF=\{z\}$ and $\imShAV=\EidAV{k}$ for some $k\geq0$.
Then $\imShAW=\EidAW{k}$ for any other connected, open neighbourhood $\Wman\subset\Vman$ of $z$.
\end{corollary}

\section{Regular extensions}\label{sect:regular-extensions}
For every $m\in\NNN$ let $\om\in\RRR^{m}$ be the origin.
Then $\omn=(\om,\on)$ if we regard $\RRR^{m+n}$ as $\RRR^{m}\times\RRR^{n}$.
Denote by $p_m:\RRR^{m+n}\to\RRR^{m}$ the natural projection to the first $m$ coordinates.

Let $\AFld$ be a germ of a vector field at $\omn$ on $\RRR^{m+n}$, and $\BFld$ be a germ of a vector field on $\RRR^{m} \equiv \RRR^{m} \times \on \subset \RRR^{m+n}$ at $\om$.
Say that $\AFld$ is a \myemph{a regular extension} of $\BFld$, if 
\begin{equation}\label{equ:Fxy__Gx_Hxy}
\AFld(x,y)=(\BFld(x),\HFld(x,y))
\end{equation}
for some $\Cinf$ germ $\HFld:\RRR^{m+n}\to\RRR^{n}$ at $\omn$.
Thus the first $m$ coordinate functions of $\AFld$ coincide with $\BFld$ and, in particular, they do not depend on the remaining $n$ coordinates.
In this case the local flows of $\AFld$ and $\BFld$ are related by the following identity:
$$
\AFlow(x,y,t)=(\BFlow(x,t), \HFlow(x,y,t)),
$$
for some $\Cinf$ map $\HFlow$.

Denote by $\FixA$ and $\FixB$ the sets of singular points of $\AFld$ and $\BFld$ respectively.
Then it follows from~\eqref{equ:Fxy__Gx_Hxy}
\begin{equation}\label{equ:FixA_FixB_for reg_ext}
 \FixA \ \subset \ \FixB \times \Vmann.
\end{equation}
Indeed, if $\AFld(x,y)=(\BFld(x),\HFld(x,y))=0$, then $\BFld(x)=0$.

Notice that if $\HFld$ is a germ of a vector field on $\RRR^{n}$ at $\on$, then the \myemph{product} $\AFld(x,y)=(\BFld(x),\HFld(y))$ of $\BFld$ and $\HFld$ is a regular extension of each of these vector fields.
Also notice that a regular extension of a regular extension is a regular extension itself.

It follows from the real version of Jordan's theorem about normal forms of matrices that every linear vector field on $\RRR^{n}$ is a regular extension of a linear vector field defined by one of the following matrices:
\begin{equation}\label{equ:minimal_irreducible_lin}
\left(\begin{smallmatrix} \lambda \end{smallmatrix}\right),
\qquad 
\left(\begin{smallmatrix} 0 & 1 \\ 0 & 0 \end{smallmatrix}\right),
\qquad 
\left(\begin{smallmatrix} a & b \\ -b & a \end{smallmatrix}\right),
\qquad \lambda,a,b\in\RRR.
\end{equation}

\begin{lemma}\label{lm:pmOmega_is_deformation}
Let $\Vmanm\subset\RRR^{m}$ and $\Vmann\subset\RRR^{n}$ be open connected subsets, $\Vmanmn=\Vmanm\times\Vmann$, $\SmParMan$ be a connected smooth manifold with $\dim\SmParMan=\dd$, $\CntParMan$ be a path-connected topological space,
\begin{equation}\label{equ:Omega_deformation_mn}
\Omega:\Vmanmn\times\SmParMan\times\CntParMan\to\RRR^{m}\times\RRR^{n}
\end{equation}
be an $\dfrm{\SmParMan}{\CntParMan}{k}$-deformation in $\EAVmn$, and  
$$
\mOmega = p_{m}\circ\Omega:\Vmanm\times\Vmann\times\SmParMan\times\CntParMan\to\RRR^{m}
$$
be the projection to the first $m$ coordinates.
Then the following statements hold true.

{\rm1)} $\mOmega$ is a $\dfrm{\Vmann\times\SmParMan}{\CntParMan}{k}$-deformation in $\EBVm$.

{\rm2)} Suppose that $\Omega$ has a continuous shift function 
\begin{equation}\label{equ:Lambda_sh_func}
\Lambda:\Vmanmn\times\SmParMan\times\CntParMan \supset A \longrightarrow\RRR,
\end{equation}
defined on some subset $A$ of $\Vmanmn\times\SmParMan\times\CntParMan$.
Then $\Lambda$ is a shift function for $\mOmega$ with respect to $\BFld$.
\end{lemma}
\begin{proof} 
1) The statement that $\mOmega$ is a $\dfrm{\Vmann\times\SmParMan}{\CntParMan}{k}$-deformation in $\Ci{\Vmanm}{\RRR^{m}}$ directly follows from Lemma\;\ref{lm:change_deform}.

Denote $\ParMan=\SmParMan\times\CntParMan$ and $\pp=(\sx,\cx)$.
We have to show that $\mOmega_{(y,\pp)}\in\EBVm$ for each $(y,\pp)\in\Vmann\times\ParMan$, i.e.
\begin{itemize}
 \item[(i)]
$\mOmega_{(y,\pp)}$ preserves orbits of $\BFld$, so for each $x\in\Vmanm$, there exists $t\in\RRR$ such that $\mOmega_{(y,\pp)}(x)=\BFlow(x,t)$;
\item[(ii)]
$\mOmega_{(y,\pp)}$ is a local diffeomorphism at each $x\in\FixB\cap\Vmanm$.
\end{itemize}

(i) Since $\Omega_{\pp}$ preserves orbits of $\AFld$, there exists $t_{x,y}\in\RRR$ such that 
$$
\Omega(x,y,\pp)=\AFlow(x,y,t_{x,y})= 
\bigl(
\BFlow(x,t_{x,y}),
\HFlow(x,y,t_{x,y})
\bigr),
$$
whence $\mOmega_{(y,\pp)}(x) = p_m\circ \Omega(x,y,\pp) = \BFlow(x,t_{x,y}).$

(ii)
Let $x\in\FixB\cap\Vmanm$, so $\BFld(x)=0$.
Then $$\AFld(x,y)=(\BFld(x),\HFld(x,y))=(0,\HFld(x,y))$$
for every $y\in\RRR^{n}$.
It follows that the orbit $\orb_{(x,y)}$ of $\AFld$ passing through $(x,y)$ is everywhere tangent to $x\times\RRR^{n}$ and therefore $\orb_{(x,y)} \subset x\times\RRR^{n}$.
Since $\Omega_{\pp}$ preserves orbits of $\AFld$, we also get that 
\begin{equation}\label{equ:Omega_xVn_xRn}
\Omega_{\pp}(x\times\Vman^{n}) \subset x\times\RRR^{n}.
\end{equation}
Consider two cases.

(a) Suppose $(x,y)$ is a singular point of $\AFld$, so $\HFld(x,y)=0$ as well.
Since $\Omega_{\pp}\in\EAVmn$, we have $\Omega_{\pp}(x,y)=(x,y)$ and $\Omega_{\pp}$ is a local diffeomorphism at $(x,y)$.
Then we get from~\eqref{equ:Omega_xVn_xRn} that the tangent space $T_{y}\Vmann=T_{(x,y)}(x\times\Vmann)$ is invariant with respect to the tangent map $T_{(x,y)}\Omega_{\pp}$ and therefore the image $T_{(x,y)}\Omega_{\pp}(T_{x}\Vman^{m})$ of the transversal space $T_{x}\Vmanm=T_{(x,y)}(\Vman^{m}\times y)$ is transversal to $T_{y}\Vman^{n}$, see Figure\;\ref{fig:tansp}.
\begin{figure}[ht]
\includegraphics[height=2.5cm]{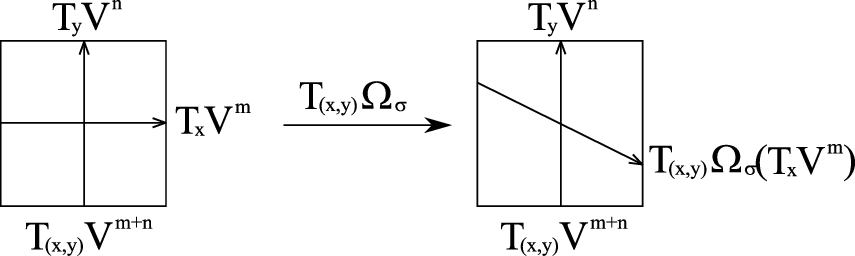} 
\caption{}
\protect\label{fig:tansp}
\end{figure}
This implies that the tangent map $T_{x}\mOmega_{(y,\pp)}=T_{x}(p_m\circ\Omega_{\pp})$ is non-degenerate, i.e. $\mOmega_{(y,\pp)}$ is a local diffeomorphism at $x$.

(b) Suppose $\HFld(x,y)\not=0$, so $(x,y)$ is a regular point of $\AFld$.
Then by Corollary\;\ref{cor:local_existence_of_sh_func} there exist neighbourhoods $\Wmanm \subset\Vmanm$ of $x$ and $\Wmann \subset\Vmann$ of $y$, and a $\Cinf$ shift function $\bfunc:\Wmanm\times\Wmann\to\RRR$ for $\Omega_{\pp}$, so
$$
\Omega_{\pp}(\bar x,\bar y) = \AFlow(\bar x,\bar y,\bfunc(\bar x,\bar y)) = (\BFlow(\bar x,\bfunc(\bar x,\bar y)),\HFld(\bar x,\bar y,\bfunc(\bar x,\bar y)),
$$
for all $(\bar x,\bar y)\in\Wmanm\times\Wmann$.
In particular, $\mOmega_{(y,\pp)}(\bar x) =\BFlow(\bar x,\bfunc(\bar x,y))$ for all $\bar x\in\Wmanm$.
Thus the restriction $\mOmega_{(y,\pp)}|_{\Wmanm}$ belongs to the image of shift map $\imSh{\BFlow,\Wmanm} \subset \EBWm$, whence $\mOmega_{(y,\pp)}$ is a local diffeomorphism at $x$.

2) Suppose $\Omega$ has a shift function\;\eqref{equ:Lambda_sh_func}.
Then the identity 
$$\Omega(x,y,\pp)=\AFlow(x,y,\Lambda(x,y,\pp)), \qquad (x,y,\pp)\in A,$$
implies
$\mOmega(x,y,\pp)=p_m\circ \Omega(x,y,\pp)=\BFlow(x,\Lambda(x,y,\pp))$.
\end{proof}

\section{Examples}\label{sect:examples}
In this section we consider examples of vector fields $\AFld$ for which relation between $\imShA$ and $\EidAFlow{k}$ is known and whose singular points satisfy assumptions (a)-(c) of Theorem\;\ref{th:main_result}.

\subsection{Vector fields on $\RRR$}\label{sect:vf_R1}
Let $\BFld(x)=\bfunc(x)\tfrac{\partial}{\partial x}$ be a vector field on $\RRR$, where $\bfunc:\RRR\to\RRR$ is a smooth function.
Suppose $\FixF=\{0\}$, so $\BFld$ has only three orbits $(-\infty,0)$, $\{0\}$, $(0,\infty)$.
This is equivalent to the assumption that $\bfunc^{-1}(0)=0$.

By the Hadamard lemma $\bfunc(x)=x\,\mu(x)$, where $\mu(x) = \int_{0}^{1}\bfunc'(tx)dt$.
Moreover, $\mu(0)=\bfunc'(0)$ and $\mu(x)\not=0$ for $x\not=0$.
Hence by Corollary\;\ref{cor:imSh_not_Eid} if $\bfunc'(0)=\mu(0)=0$, then $\imShB\not=\EidBFlow{\infty}$.

If $\bfunc'(0)\not=0$, then $\mu\not=0$ on all of $\RRR$ and by Lemma\;\ref{lm:Sh-EA-under-reparametrization} instead of $\BFld$ we can consider the \myemph{linear} vector field $\AFld(x)=x\tfrac{\partial}{\partial x}$.

It is shown in\;\cite{Maks:TA:2003} that $\imShA=\EidAFlow{0}$ for any linear vector field on $\RRR^n$.
We will now discuss these results.

\subsection{}\label{sect:shift-func-linear-R1}
Let $\AFld(x)= v x\frac{\partial}{\partial x}$ be the linear vector field on $\RRR$.
Then it generates the following flow $\AFlow(x,t)=e^{v t}x$.
Let $\Vman$ be an open neighbourhood of $0$ and $\Omega:\Vman\times\SmParMan\times\CntParMan\to\RRR$ be an $\dfrm{\SmParMan}{\CntParMan}{k}$-deformation in $\EAV$.
Then $\Omega_{(\sx,\cx)}(0)=0$, and by the Hadamard lemma 
$$\Omega(x,\sx,\cx)=x \smallint_{0}^{1}\tfrac{\partial\Omega}{\partial x}(tx,\sx,\cx)dt.$$
Since $\Omega_{(\sx,\cx)}\in\EAV$, we have that $\tfrac{\partial\Omega}{\partial x}(x,\sx,\cx)>0$, i.e. $\Omega_{(\sx,\cx)}$ is a local diffeomorphism at $0$.
Hence we can define the following function 
$\Lambda:\Vman\times\SmParMan\times\CntParMan\to\RRR$ by
$$
\Lambda(x,\sx,\cx) = \tfrac{1}{v}\ln\smallint_{0}^{1}\tfrac{\partial\Omega}{\partial x}(tx,\sx,\cx)dt.
$$
Then $\Lambda$ is a $\dfrm{\SmParMan}{\CntParMan}{k-1}$-deformation being a shift function for $\Omega$, i.e. $\Omega(x,\sx,\cx)=e^{v\cdot \Lambda(x,\sx,\cx)}$, c.f.\;\cite[Eq.\;(23)]{Maks:TA:2003}.
In particular, we get $\imShAV=\EAV$.

\subsection{}\label{sect:shift-func-linear-nilp-R2}
Consider the following linear vector field on $\RRR^2$:
$$
\AFld(x,y)=(\tfrac{\partial}{\partial x}, \tfrac{\partial}{\partial y})
\left(\begin{smallmatrix} 0 & 1 \\ 0 & 0 \end{smallmatrix}\right)
\left(\begin{smallmatrix} x  \\ y \end{smallmatrix}\right) =
y\tfrac{\partial}{\partial x}.
$$
It generates the flow $\AFlow(x,y,t) = (x+yt, y)$.
Let $\Vman$ be an open neighbourhood of $\orig$, and $\Omega=(\Omega_1,\Omega_2):\Vman\times\SmParMan\times\CntParMan\to\RRR$ be an $\dfrm{\SmParMan}{\CntParMan}{k}$-deformation in $\EAV$.
Then $\Omega_2(x,y,\sx,\cx)\equiv y$ and $\Omega_1(x,0,\sx,\cx)\equiv x$.
Hence by the Hadamard lemma
$$\Omega_{1}(x,y,\sx,\cx)-x =  y \smallint_{0}^{1}\tfrac{\partial\Omega_1}{\partial y}(x,ty,\sx,\cx)dt.$$
Define the function 
$\Lambda:\Vman\times\SmParMan\times\CntParMan\to\RRR$ by
$$
\Lambda(x,y,\sx,\cx) = \smallint_{0}^{1}\tfrac{\partial\Omega_1}{\partial y}(x,ty,\sx,\cx)dt.
$$
Then $\Lambda$ is a $\dfrm{\SmParMan}{\CntParMan}{k-1}$-deformation being a shift function for $\Omega$, that is 
$\Omega(x,y,\sx,\cx)=(x+ y\cdot\Lambda(x,y,\sx,\cx), y)$, c.f.\;\cite[Eq.\;(26)]{Maks:TA:2003}.
Again we get $\imShAV=\EAV$.

\subsection{}\label{sect:shift-func-linear-rotation-R2}
Let $\omega=a+bi\not=0$ be a complex number, $\AFlow(z,t)=e^{\omega\, t} z$ be the corresponding linear flow on $\CCC$, and $\Vman\subset\CCC$ be an open neighbourhood of $\orig$.
By definition every $\Omega\in\EAV$ is a local diffeomorphism at $\orig$.

Let $\Omega:\Vman\times\RRR^{\dd}\to\Mman$ be a $(\RRR^{\dd},\Cinf)$-deformation in $\EAV$ such that $\Omega_{\sx}$ preserves orientation at $\orig$ for each $\sx\in\RRR^{\dd}$.
Then by \cite[Lm.\;31]{Maks:TA:2003} there exists a $\Cinf$ function $\gamma:\Vman\times\RRR^{\dd}\to\CCC\setminus\orig$ such that $\Omega(z,\sx)=z\cdot\gamma(z,\sx)$.
Hence, \cite[Eqs.(24),(25)]{Maks:TA:2003}, the following function $\Lambda:\Vman\times\RRR^{\dd}\to\RRR$ defined by 
$$
\Lambda(z)=
\begin{cases}
\tfrac{1}{2a} \ln|\gamma(z,\sx)|^2, & a\not=0, \\
\tfrac{1}{b} \arg(\gamma(z,\sx)), & a=0, 
\end{cases}
$$
is a $\Cinf$ shift function for $\Omega$.
In the second case $\Lambda$ is defined up to a constant summand $2\pi k/n$.

\begin{remark}\rm
In \S\ref{sect:shift-func-linear-R1} and \S\ref{sect:shift-func-linear-nilp-R2} a shift function $\Lambda$ of an $\dfrm{\SmParMan}{\CntParMan}{k}$-defor\-mation $\Omega$ is expressed via partial derivatives of $\Omega$.
Therefore it looses its ``smoothness in $\tau$'' by $1$ though it remains $\Cinf$ in $\sx$.
The main tool of constructing $\Lambda$ is the Hadamard lemma.

On the other hand in \S\ref{sect:shift-func-linear-rotation-R2} the proof of mentioned above ``division lemma'' \cite[Lm.\;31]{Maks:TA:2003} uses \myemph{all} the partial derivatives of $\Omega$ for the proof of smoothness of $\gamma$.
Hence if $\Omega$ were an $\dfrm{\SmParMan}{\CntParMan}{k}$-deformation, then we could not guarantee that the obtained function $\Lambda$ is even continuous in $\cx$ though for each $\cx\in\CntParMan$ the function $\Lambda_{\cx}$ is $\Cinf$ on $\Vman$, c.f. Definition\;\ref{defn:ES_dk}.
Such effects are typical for divisibility by smooth functions, c.f.\cite[p.~93. Eq.~(2)]{Mather_1:AnnMath:1968} and~\cite{MostowShnider:TrAMS:1985}.
\end{remark}

\subsection{Properties $\EP{\ld}$}\label{sect:prop_En}
In\;\cite[Defn.\;24]{Maks:TA:2003} for every $\ld\geq0$ the following ``extension'' property $\EP{\ld}$ for a singular point $z$ of a vector field $\AFld$ was introduced (see also Remark\;\ref{rem:Maks:TA:2003-gaps}).
We will now present a slight modification of this property which will be useful for further considerations.

\begin{definition}\label{defn:prop_Ed}
Let $z\in\FixF$ be such that $\FixF$ is nowhere dense at $z$.
Say that $\AFld$ has \myemph{property $\EP{\ld}$}, $(\ld\geq0)$, at $z$ if for any 
\begin{itemize}
\item
open neighbourhood $\Vman$ of $z$,
\item
a smooth manifold $\SmParMan^{\ld}$ of dimension $\dim\SmParMan^{\ld}=\ld$,
\item
an $(\SmParMan^{\ld},\Cinf)$ deformation $\Omega:\Vman\times\SmParMan^{\ld}\to\Mman$ in $\EAV$, and
\item
a $\Cinf$ shift function $\Lambda:(\Vman\setminus\FixF)\times\SmParMan^{\ld}\to\RRR$ for $\Omega$, i.e. 
\begin{equation}\label{equ:Lambda_is_shift_func}
\Omega(x,\sx)=\AFlow(x,\Lambda(x,\sx)), \qquad (x,\sx)\in(\Vman\setminus\FixF)\times\SmParMan^{\ld},
\end{equation}
\end{itemize}
there exists a neighbourhood $\Wman\subset\Vman$ of $z$ such that $\Lambda$ extends to a $\Cinf$ function $\Lambda:\Wman\times\SmParMan^{\ld}\to\RRR$.
\end{definition}

Thus in order to verify $\EP{\ld}$ one should \myemph{smoothly resolve} \eqref{equ:Lambda_is_shift_func} with respect to $\Lambda$ on a neighbourhood of $z\times\SmParMan^{\ld}$ using the assumption that $\Omega_{\sx}$ preserves orbits of $\AFld$ and is a local diffeomorphism at each $x\in\FixF\cap\Vman$.
In particular, vector fields of\;\S\S\ref{sect:shift-func-linear-R1}-\ref{sect:shift-func-linear-rotation-R2} evidently have $\EP{\ld}$ for all $\ld\geq0$.
In the first two cases this is trivial, and in the last case it should be noted that if $\Omega_{\sx}$ has a shift function on $\Vman\setminus\FixF=\Vman\setminus\orig$, then it should preserve orientation at $\orig$.

Property $\EP{0}$ allows to prove that $\imShA=\EidAFlow{0}$.

\begin{lemma}\label{lm:prop_E0}{\rm\cite[Th.\;25]{Maks:TA:2003}}
Suppose that $\FixF$ is nowhere dense and $\AFld$ has property $\EP{0}$ at each $z\in\FixF$.
Then $\imShA=\EidAFlow{0}$.
\end{lemma}

As a direct consequence of Lemma\;\ref{lm:pmOmega_is_deformation}, we obtain that properties $\EP{\ld}$ for $\ld>0$ allow to prove $\EP{0}$, and therefore to establish $\imShA=\EidAFlow{0}$, for regular extensions.

\begin{lemma}\label{lm:En_heredity}{\rm\cite[Lm.\;29]{Maks:TA:2003}}, (see also Remark\;\ref{rem:Maks:TA:2003-gaps} and Lemma\;\ref{lm:pmOmega_is_deformation})
Let $\BFld$ (resp. $\AFld$) be a germ of a vector field on $\RRR^{n}$ at $\on$ (resp. on $\RRR^{m+n}$ at $\omn$).
Suppose $\AFld$ is a regular extension of $\BFld$.
If $\BFld$ has property $\EP{\ld+n}$ at $\om$, then $\AFld$ has property $\EP{\ld}$ at $\omn$.
\end{lemma}

\begin{definition}\label{defn:reLT_vector fields}
Let $\AFld$ be a vector field on a manifold $\Mman$ and $z\in\FixF$.
Say that $z$ is of type \reLT\ for $\AFld$, if the germ of $\AFld$ at $z$ is a regular extension of some linear vector field on $\RRR^{n}$ at $\on$ for some $n>0$.
\end{definition}

\begin{corollary}\label{cor:regext_lin_Ed}{\rm\cite[Th.\;27]{Maks:TA:2003}}.
Every point of type \reLT\ has property $\EP{\ld}$ for each $\ld\geq0$.
\end{corollary}
\begin{proof}
As noted above, every linear vector field is a regular extension of some vector field considered in \;\S\S\ref{sect:shift-func-linear-R1}-\ref{sect:shift-func-linear-rotation-R2}, see also\;\eqref{equ:minimal_irreducible_lin}.
Then our statement follows from Lemma\;\ref{lm:En_heredity}.
\end{proof}

The problem of linearization of a vector field in a neighbourhood of a singular point was extensively studied, see e.g.\;\cite{Siegel:1952, Sternberg:AmJM:1957, Venti:JDE:1966, KondratevSamovol:MZ:1973, Blackmore:JDE:1973}.

\begin{corollary}\label{cor:regext_lin_imSh_EidA0}{\rm\cite[Th.\;1(A)]{Maks:TA:2003}}.
Suppose every $z\in\FixF$ is of type \reLT\ for $\AFld$.
Then $\imShA=\EidAFlow{0}.$
\end{corollary}

\subsection{Reduced Hamiltonian vector fields.}\label{sect:reduced_hvf}
In\;\cite{Maks:CEJM:2009} the author presented a class of examples of highly degenerate vector fields $\AFld$ on $\RRR^n$ with singularity at $\orig$ for which \myemph{$\imShA$ coincides at $\orig$ with $\EidAFlow{1}$ formally}.
That is for each $\dif\in\EidAFlow{1}$ there exists $\afunc\in\Ci{\RRR^{n}}{\RRR}$ such that 
$j^{\infty}\dif(\orig)=j^{\infty}\ShA(\afunc)(\orig)$.
Every such vector field (and even its initial non-zero jet at $\orig$) turned out to be non-divisible by smooth functions, c.f. Corollary\;\ref{cor:imSh_not_Eid}.

Moreover, in the following special case it can be said more.

Let $g:\RRR^2\to\RRR$ be a homogeneous polynomial of degree $p\geq 2$, so
\begin{equation}\label{equ:g_homog_poly}
g=L_1^{l_1} \cdots L_{a}^{l_a}  \cdot Q_1^{q_1} \cdots Q_{b}^{q_b}, 
\end{equation}
where $L_i$ is a non-zero linear function, $Q_j$ is an irreducible over $\RRR$ (definite) quadratic form, $l_i,q_j\geq 1$, $L_i/L_{i'}\not=\mathrm{const}$ for $i\not=i'$, and $Q_j/Q_{j'}\not=\mathrm{const}$ for $j\not=j'$.
Put
$$
D = L_1^{l_1-1} \cdots L_{a}^{l_a-1}  \cdot Q_1^{q_1-1} \cdots Q_{b}^{q_b-1}.
$$
Then $g=L_1 \cdots L_{a} \cdot Q_1 \cdots Q_{b}\cdot D$ and it is easy to see that \myemph{$D$ is the greatest common divisor of the partial derivatives $g'_{x}$ and $g'_{y}$}.
The following \myemph{polynomial} vector field on $\RRR^2$:
$$
\AFld(x,y)=-(g'_{y}/D)\,\tfrac{\partial}{\partial x} + (g'_{x}/D)\, \tfrac{\partial}{\partial y}
$$ 
will be called the \myemph{reduced Hamiltonian} vector field of $g$.
In particular, if $g$ has no multiple factors, i.e. $l_i=q_j=1$ for all $i,j$, then $D\equiv1$ and $\AFld$ is the usual \myemph{Hamiltonian} vector field of $g$.

Notice that $\AFld(g)\equiv0$ and the coordinate functions of $\AFld$ are relatively prime homogeneous polynomials of degree $\deg\AFld=a+2b-1$.

If $\deg\AFld=1$, so $\AFld$ is linear, then by Corollary\;\ref{cor:regext_lin_imSh_EidA0}, $\imShA=\EidAFlow{0}$.

Suppose $\deg\AFld\geq2$, so $a+2b\geq3$.
Then we will distinguish the following two cases, see Figure\;\ref{fig:H_type}.

{\bf Case \RHE}: $a=0$ and $b\geq2$.
Thus $g=Q_1\cdots Q_{b}$ is a product of at least two distinct irreducible quadratic forms, and, in particular, the origin $\orig\in\RRR^2$ is a global \myemph{extreme} of $g$

{\bf Case \RHS}: all other cases, so either $a\geq3$ and $b=0$, or $a\geq1$ but $b\geq1$.
Then $\orig\in\RRR^2$ is a \myemph{saddle} critical point of $g$.

\begin{figure}[ht]
\begin{tabular}{ccc}
\includegraphics[height=1.2cm]{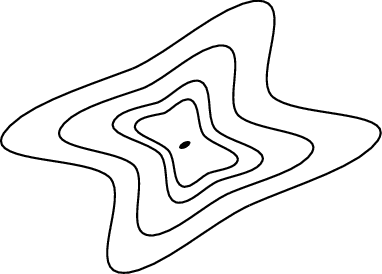} & \qquad \qquad \qquad &
 \includegraphics[height=1.2cm]{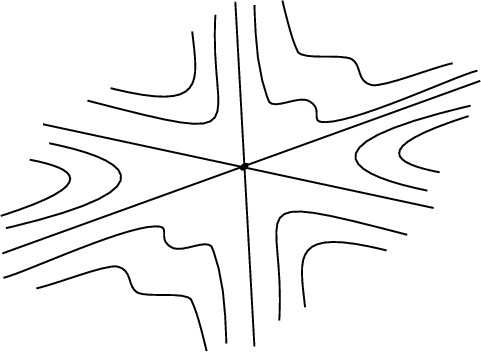}  \\
\RHE & & \RHS 
\end{tabular}
\caption{}
\protect\label{fig:H_type}
\end{figure}

It is shown in\;\cite{Maks:MFAT:2009} that in the case \RHE\ $\imShA\not=\EidAFlow{0}$.
By Lemma\;\ref{lm:prop_E0} this implies that \RHE-singularities do not satisfy $\EP{0}$.

The last statement can also be verified by another arguments.
Let $\theta:\RRR^2\setminus\orig\to\RRR$ be the function associating to every $z\not=\orig$ its period with respect to $\AFld$.
Then $\theta$ is a $\Cinf$ shift function on $\RRR^2\setminus\orig$ for the identity map $\id_{\RRR^2}$.
But since $j^1\AFld(\orig)=0$, we have that $\lim\limits_{z\to\orig}\theta(z)=\infty$, see Example\;\ref{exmp:top_sing}.
Hence $\theta$ can not be even continuously extended to all of $\RRR^2$.
Therefore the origin $\orig\in\RRR^2$ does not satisfy $\EP{0}$.

\begin{definition}
Let $\AFld$ be a vector field on a manifold $\Mman$ and $z\in\FixF$.
We say that $z$ is of type \reRHE, resp. \reRHS, for $\AFld$, if the germ of $\AFld$ at $z$ is a regular extension of some reduced Hamiltonian vector field of a homogeneous polynomial~\eqref{equ:g_homog_poly} of the case \RHE, resp. \RHS.

In either the cases we say that $z$ is of type \reRH.
\end{definition}

\begin{example}\rm
A singular point of a vector field can belong to distinct types.
For instance, consider the following vector field on $\RRR^5$:
$$
\AFld(a,b,c,d,e) \ = \  (2a, \ \ - 4c^3, \ 4b^3, \ \ -d^2-2de, \ 3de+2e^2)
$$
Then $\AFld$ is a product of the linear vector field $2a\tfrac{\partial}{\partial a}$ and reduced Hamiltonian vector fields of polynomials $b^4+c^4$ and $d^3\,e + d^2e^2$.
Whence the origin $\orig\in\RRR^5$ belongs to each of the types \reLT, \reRHE, and \reRHS.
\end{example}

As an application of \cite{Maks:CEJM:2009} and the results of the present paper we will prove in next section the following theorem which extends Corollary~\ref{cor:regext_lin_imSh_EidA0} to vector fields with singularities of type \reRH.
This theorem is also a ``global'' variant of \cite[Th.\;11.1]{Maks:CEJM:2009}.
\begin{theorem}\label{th:LT_RH}
Let $\AFld$ be a vector field on a manifold $\Mman$ tangent to $\partial\Mman$ and such that every $z\in\FixF$ belongs to one of the types \reLT\ or \reRH.
Then $\imShA=\EidAFlow{1}$.
Moreover, $\AFld$ satisfies condition $\GSF$, i.e. if $\BFld$ is another vector field each of whose orbits is contained in some orbit of $\AFld$, then there exists a $\Cinf$ function $\afunc:\Mman\to\RRR$ such that $\BFld=\afunc\,\AFld$.

Moreover, if, in addition, every $z\in\FixF$ belongs to one of the types \reLT\ or \reRHS, then $\imShA=\EidAFlow{0}$.
\end{theorem}

\section{Properties $\ESDL{\ld}{\dd}{k}$}\label{sect:ES_dk_properties}
In this section we introduce another series of properties $\ESDL{\ld}{\dd}{k}$ being weaker than $\EP{\ld}$, but similarly to them property $\ESDL{0}{1}{k}$ for singular points of $\AFld$ guarantees $\imShA=\EidAFlow{k}$, while $\ESDL{\ld}{\dd}{k}$ implies $\ESDL{0}{\dd}{k}$ for regular extensions.
We also describe $\ESDL{\ld}{\dd}{k}$ for singularities of types \reRH.
This will allow to prove Theorem\;\ref{th:LT_RH}.

For $\ld\geq0$ let $\RRR^{\ld}_{+}=\{x_{\ld}\geq0\}$ be closed upper half space in $\RRR^{\ld}$.

\begin{definition}\label{defn:ES_dk}
Suppose that $\FixF$ is nowhere dense at $z\in\FixF$ and let $0\leq\ld<\infty$, $0\leq\dd<\infty$, and $0\leq k\leq \infty$.
Say that $\AFld$ has \myemph{extensions of shift functions under $k$-deformations} property $\ESDL{\ld}{\dd}{k}$ at $z$ if the following holds true: let
\begin{itemize}
 \item 
$\Vman$ be a connected, open neighbourhood of $z$, 
 \item 
$\SmParMan^{\ld}$ be a either $\RRR^{\ld}$ or $\RRR^{\ld}_{+}$,
 \item 
$\Omega:\Vman\times\SmParMan^{\ld}\times I^{\dd}\to\Mman$ be an $\dfrm{\SmParMan^{\ld}}{I^{\dd}}{k}$-deformation in $\EAV$, such that for some $(\sx_0,\cx_0)\in\SmParMan^{\ld}\times I^{\dd}$ the map $\Omega_{(\sx_0,\cx_0)}$ has a $\Cinf$ function $\afunc$ on all of $\Vman$, and
 \item 
$\Lambda:(\Vman\setminus\FixF)\times\SmParMan^{\ld}\times I^{\dd}\to\RRR$ be a unique continuous shift function for $\Omega$ such that $\Lambda_{(\sx_0,\cx_0)}=\afunc$;
\end{itemize}
then there exists a neighbourhood $\Wman\subset\Vman\times\SmParMan^{\ld}$ of $(z,\sx_0)$ such that for each $\cx\in I^{\dd}$ the function $\Lambda_{\cx}$ smoothly extends to $\Wman$.

For $\dd=0$ the number $k$ does not matter, therefore in this case we will denote the corresponding property by $\ESDL{\ld}{0}{-}$.
Also notice that $\ESDL{0}{0}{-}$ is a tautology, therefore we usually assume that $\ld+\dd\geq1$.
\end{definition}

\begin{remark}\rm
a) Definition\;\ref{defn:ES_dk} \myemph{does not require} for $\Lambda$ to be continuous on $\Wman\times\SmParMan^{\ld}\times I^{\dd}$.
 
b) Also notice that $\ESDL{\ld}{\dd}{k}$ is a property of the germ of $\AFld$ at $z$, i.e.\! if $\Uman$ is an open neighbourhood of $z$, then $\AFld$ has $\ESDL{\ld}{\dd}{k}$ at $z$ if and only if the restriction $\AFld|_{\Uman}$ has this property at $z$.
\end{remark}

The following lemma is a direct consequence of definitions and we left it for the reader.

\begin{lemma}\label{lm:impl_between_E_ESDL}
For $0\leq \ld\leq \ld'<\infty$, $0\leq \dd\leq\dd'<\infty$, and $0\leq k\leq k'\leq \infty$ the following implications hold true:

\begin{itemize}
 \item[(a)]
$\EP{\ld'} \ \Rightarrow \ \EP{\ld}$,
 \item[(b)]
$\EP{\ld} \ \Rightarrow \ \ESDL{\ld}{\dd}{k}$,
 \item[(c)]$\ESDL{\ld'}{\dd'}{k} \ \Rightarrow \ \ESDL{\ld}{\dd}{k'}$.
\end{itemize}
Thus $\ESDL{1}{0}{-}$ and $\ESDL{0}{1}{\infty}$ are the weakest properties. \qed
\end{lemma}

Evidently, Theorem\;\ref{th:main_result} claims that either of its conditions (a)-(c) implies $\ESDL{0}{\dd}{k}$ for $z$ for all $\dd\geq0$.
Moreover, the proof of Theorem\;\ref{th:suffcond_for_imShA_EidAk} only uses property $\ESDL{0}{\dd}{k}$.
Hence the following result being also an analogue of \cite[Th.\;25]{Maks:TA:2003} (see Lemma\;\ref{lm:prop_E0}) holds true.
We leave the details for the reader.

\begin{theorem}\label{th:suffIMEk}
Suppose $\FixF$ is nowhere dense in $\Vman$ and there exists $k\in\{0,\ldots,\infty\}$ such that $\AFld$ has property $\ESDL{0}{1}{k}$ at every $z\in\FixF\cap\Vman$.
Then 
\begin{equation}\label{equ:imShAV_EidAVk}
\imShAV=\EidAV{k},
\end{equation}
In particular, if $\Vman=\Mman$, then $\imShA=\EidAFlow{k}$
\end{theorem}

\begin{lemma}\label{lm:param-rigidity}{\rm c.f.\;\cite[Th.\;4.1]{Maks:CEJM:2009}}
Suppose $\AFld$ has the weakest properties $\ESDL{1}{0}{-}$ and $\ESDL{0}{1}{\infty}$ at each of its singular points, then $\AFld$ has $\GSF$ as well, see \S\ref{sect:prop-GSF}.
\end{lemma}
\begin{proof}
It is shown in\;\cite[Th.\;4.1]{Maks:CEJM:2009} that if for each $z\in\FixF$ there exists a neighbourhood $\Vman$ such that $\imShAV=\EidAV{\infty}$ and the shift map $\ShAV$ satisfies the so-called \myemph{smooth path-lifting condition}, see\;\cite[Defn.\;4.2]{Maks:CEJM:2009}, then $\AFld$ has property $\GSF$.

The first assumption $\imShAV=\EidAV{\infty}$ can be satisfied due to property $\ESDL{0}{1}{\infty}$ for singular points of $\AFld$ and Theorem\;\ref{th:suffIMEk}.
Moreover, \myemph{smooth path-lifting condition} is the same as $\ESDL{1}{0}{-}$.
\end{proof}

The next statement describes in such a way $\ESDL{\ld}{\dd}{k}$ is inherited with respect to regular extensions, c.f. Lemma\;\ref{lm:En_heredity}.

\begin{lemma}\label{lm:inher_ESDL}
Let $\BFld$ (resp. $\AFld$) be a germ of a vector field on $\RRR^{n}$ at $\on$ (resp. on $\RRR^{m+n}$ at $\omn$).
Suppose $\AFld$ is a regular extension of $\BFld$, and $\AFld(\omn)=0$, so $\BFld(\om)=0$ as well.
If $\BFld$ has property $\ESDL{\ld+n}{\dd}{k}$ at $\om$, then $\AFld$ has property $\ESDL{\ld}{\dd}{k}$ at $\omn$.
\end{lemma}
\begin{proof}
Let $\Vmanm$ and $\Vmann$ be open connected neighbourhoods of $\om$ and $\on$ respectively.
Then $\Vmanmn=\Vmanm\times\Vmann$ is a connected, open neighbourhood of $\omn$.

Let $\SmParMan^{\ld}=\RRR^{\ld}$ or $\RRR^{\ld}_{+}$, $(\sx_0,\cx_0)\in\SmParMan\times I^{\dd}$, $\Omega:\Vmanmn\times\SmParMan^{\ld}\times I^{\dd}\to\RRR^{m+n}$ be an $\dfrm{\SmParMan^{\ld}}{I^{\dd}}{k}$-deformation in $\EAVmn$, such that $\Omega_{(\sx_0,\cx_0)}=\ShAVmn(\afunc)$ for some $\afunc\in\Ci{\Vmanmn}{\RRR}$, and $\Lambda:(\Vmanmn\setminus\FixA)\times\SmParMan^{\ld}\times  I^{\dd}\to\RRR$ be a unique shift function for $\Omega$ such that $\Lambda_{(\sx_0,\cx_0)}=\afunc$.

By Lemma~\ref{lm:pmOmega_is_deformation} the projection 
$$
\mOmega=p_{m}\circ\Omega:\Vmanm\times\Vmann\times\SmParMan^{\ld} \times I^{\dd} \to \RRR^{m}
$$
is a $\dfrm{\Vmann\times\SmParMan^{\ld}}{I^{\dd}}{k}$-deformation in $\EBVm$ and $\Lambda$ is a shift function for $\mOmega$.
Then by $\ESDL{\ld+n}{\dd}{k}$ for $\om$ there exists a neighbourhood $\Wman$ of $(\om,\on,\sx_0)$ in $\Vmanm\times\Vmann\times\SmParMan^{\ld}$ such that for each $\cx\in I^{\dd}$ the function $\Lambda_{\cx}$ smoothly extends to $\Wman$.
This implies $\ESDL{\ld}{\dd}{k}$ for $\omn$.
\end{proof}

\newcommand\iShAV{\lambda}
\begin{lemma}\label{lm:Ed_and_ESDL_for_RH}
Every point of type \reRHS\ has property $\EP{\dd}$ for all $\dd\geq0$.
Every point of type \reRHE\ has property $\ESDL{\ld}{\dd}{k}$ for all $\ld,\dd\geq0$ and $k\geq1$, but in general it does not satisfy $\ESDL{\ld}{\dd}{0}$ with $\ld+\dd\geq1$.
\end{lemma}
\begin{proof}
Let $g$ be a homogeneous polynomial\;\eqref{equ:g_homog_poly}, $\AFld$ be the corresponding reduced Hamiltonian vector field of $g$, and $\Vman$ be an open connected neighbourhood of $\orig\in\RRR^2$.
Suppose that $\deg\AFld\geq2$.
It is shown in\;\cite{Maks:hamv2,Maks:CEJM:2009}, see\;\cite[Th.\;11.1]{Maks:CEJM:2009}, that in this case

1) $\EidAV{1} = \{ \dif\in\EAV \ : \ j^1\dif(\orig)=\id\}$, and

2) there exists a map $\iShAV:\EidAV{1}\to\Ci{\Vman}{\RRR}$ with the following properties.

\begin{enumerate}
 \item[(2a)]
$\ShAV\circ\iShAV = \id(\EidAV{1})$, that is for each $\dif\in\EidAV{1}$ we have that $\dif(z)=\AFlow(z,\iShAV(\dif)(z))$.
Thus $\iShAV$ is the inverse of $\ShAV$.
In particular, we obtain that $\EidAV{1}=\imShAV$.
 \item[(2b)]
the map $\iShAV$ \myemph{preserves smoothness} in the following sense: for any $(\SmParMan,\Cinf)$-deformation $\Omega:\Vman\times\SmParMan\to\RRR^2$ in $\EidAV{1}$ the function $\Lambda:\Vman\times\SmParMan\to\RRR$ defined by $\Lambda(z,\sx) = \iShAV(\Omega_{\sx})(z)$ is also $\Cinf$.
\end{enumerate}

Now we can complete our lemma.
Denote $\VmanS=\Vman\setminus\FixF$.
Let also $\SmParMan^{\ld}$ be any smooth manifold of dimension $\ld$.

\RHS.
Let $\Omega:\Vman\times\SmParMan^{\ld}\to\RRR^2$ be an $(\SmParMan^{\ld},\Cinf)$-deformation in $\EAV$ and $\Lambda:\VmanS\times\SmParMan^{\ld}\to\RRR$ be any shift function for $\Omega$.
It is shown in\;\cite{Maks:MFAT:2009} that $\EidAV{1}=\EAV$ for the case \RHS, so $\Omega$ is a deformation in $\EidAV{1}=\imShAV$.

Put $\afunc_{\sx}=\iShAV(\Omega_{\sx})$.
Then $\Lambda_{\sx}$ and $\afunc|_{\sx}|_{\VmanS}$ are two shift functions for $\Omega_{\sx}$ on $\VmanS$.
Since $\orig$ is a ``saddle'' point, we obtain that $\VmanS$ contains non-closed orbits of $\AFld$, see Figure\;\ref{fig:H_type}.
Hence $\ShAVS$ is non-periodic.
Therefore $\Lambda_{\sx}=\afunc|_{\sx}|_{\VmanS}$, so $\Lambda_{\sx}$ smoothly extends to $\Vman$.

Thus we obtain a function $\Lambda:\Vman\times\SmParMan^{\ld}\to\RRR$ which can be defined by $\Lambda(z,\sx) = \iShAV(\Omega_{\sx})(z)$.
Then by (2b) $\Lambda$ is $\Cinf$.
This proves $\EP{\dd}$ for $\orig$.

\RHE.
Let $\Omega:\Vman\times\SmParMan^{\ld}\times I^{\dd}\to\RRR^2$ be an $\dfrm{\SmParMan^{\ld}}{ I^{\dd}}{1}$-deformation in $\EidAV{1}=\ShAV$, $(\sx_0,\cx_0)\in\SmParMan^{\ld}\times I^{\dd}$, $\afunc\in\Ci{\Vman}{\RRR}$ be a $\Cinf$ shift function for $\Omega_{(\sx_0,\cx_0)}$ and $\Lambda:\VmanS\times\SmParMan^{\ld}\times I^{\dd}\to\RRR$ be a unique continuous shift function for $\Omega$ such that $\Lambda_{(\sx_0,\cx_0)}=\afunc$.

Notice that in the case \RHE\ the origin $\orig$ is a strong $\PN$-point for $\AFld$ and also  $j^1\AFld(\orig)=0$, so it satisfies even both assumptions (b) and (c) of Theorem\;\ref{th:main_result}.
Hence for each $(\sx,\cx)\in\SmParMan^{\ld}\times I^{\dd}$ the function $\Lambda_{(\sx,\cx)}$ smoothly extends to all of $\Vman$ and coincides with a unique shift function $\afunc_{(\sx,\cx)}=\iShAV(\Omega_{(\sx,\cx)})$ for $\Omega_{(\sx,\cx)}$.

Then for each $\cx\in I^{\dd}$ the obtained function $\Lambda_{\cx}:\Vman\times\SmParMan^{\ld}\to\RRR$ is defined by $\Lambda_{\cx}(z,\sx) = \iShAV(\Omega_{(\sx,\cx)})(z)$.
Hence by (2b) $\Lambda_{\cx}$ is $\Cinf$ on $\Vman\times\SmParMan^{\ld}$.
This proves $\ESDL{\ld}{\dd}{1}$ for $\orig$.

To show that $\orig$ does not satisfy $\ESDL{\ld}{\dd}{0}$ consider the mapping $\dif:\RRR^2\to\RRR^2$ defined by $\dif(z)=-z$.
Since $g$ is a product of definite quadratic forms, it follows that $g\circ\dif=g$, whence $\dif\in\EAV$.
It is shown in\;\cite{Maks:MFAT:2009} that $\dif\in\EidAV{0}\setminus\EidAV{1}$, therefore $\imShAV\not=\EidAV{0}$.

Let us briefly recall the idea of proof.
First it is not hard to show that the set of $1$-jets at $\orig$ of elements of $\EAV$ constitute a finite (actually either dihedral or cyclic) subgroup of $\mathrm{GL}(2,\RRR)$.
Then, using analogue of a Alexander trick, it was constructed a $0$-homotopy $\dif_{t}$ between $\dif_0=\id_{\RRR^2}$ and $\dif_1=\dif$ in $\EAV$.
Neither of such homotopies can be a $1$-homotopy.
Indeed, $j^1\dif_0(\orig)=\id$, while $j^1\dif_1(\orig)=-\id$.
But as just noted, $j^1\dif_t(\orig)$ can take only finitely many values, whence $j^{1}\dif_t$ is not continuous in $t$, so $(\dif_t)$ is not a $1$-homotopy.
\end{proof}

\subsection{Proof of Theorem~\ref{th:LT_RH}}\label{sect:proof:th:LT_RH}
By Corollary\;\ref{cor:regext_lin_Ed} and Lemmas\;\ref{lm:impl_between_E_ESDL}, \ref{lm:inher_ESDL}, and \ref{lm:Ed_and_ESDL_for_RH} every point of type \reLT\ and \reRHS\ has property $\EP{0}$, while every point of type \reRHE\ has property $\ESDL{0}{1}{1}$.
Then Theorem\;\ref{th:suffIMEk} implies the first statement of out theorem about $\imShAV$.

Moreover, again by Lemma\;\ref{lm:impl_between_E_ESDL} every singular point of $\AFld$ has properties $\ESDL{1}{0}{-}$ and $\ESDL{0}{1}{\infty}$, whence by Lemma\;\ref{lm:param-rigidity} $\AFld$ has property $\GSF$.
\qed

\section{Acknowledgment}
I would like to thank the anonymous referee for very useful comments which allow to improve the exposition of the paper: in particular, for the information about the papers\;\cite{Montgomery:AJM:1937, Epstein:AnnMath:1972, Sullivan:PMIHES:1976} and for pointing out to the standard definition of parameter rigidity.

\bibliographystyle{amsplain}
\bibliography{biblio}

\end{document}